\documentclass[sigconf, nonacm, dvipsnames]{acmart} %
\acmConference[Woodstock '18]{Woodstock '18: ACM Symposium on Neural
	Gaze Detection}{June 03--05, 2018}{Woodstock, NY}

\usepackage[T1]{fontenc}
\usepackage{etoolbox}

\usepackage{relsize}
\usepackage{enumerate}
\usepackage{booktabs}
\usepackage[format = plain]{caption}
\usepackage{subcaption}
\usepackage{algorithm}
\usepackage{algorithmic}
\usepackage[autostyle=true]{csquotes}
\usepackage{verbatim}
\usepackage{tikz}
\usepackage{tikz-cd}
\usepackage{microtype}
\usepackage{pdfpages}
\usepackage[british]{babel}
\usepackage{mathtools}
\usepackage{amsthm}
\usepackage{siunitx}
\usepackage[capitalize,noabbrev]{cleveref}

\newtheorem{theorem}{Theorem}[section]
\newtheorem*{theorem*}{Theorem}

\newtheorem{lemma}[theorem]{Lemma}

\theoremstyle{definition}
\newtheorem{definition}[theorem]{Definition}

\theoremstyle{remark}

\usepackage[textsize=tiny]{todonotes}

\newcommand{\SC}{\mathcal{S}}
\newcommand{\mS}{\mathbb{S}}
\newcommand{\N}{\mathbb{N}}

\newcommand{\TDA}{\textsmaller{TDA}}
\newcommand{\tSC}{\textsmaller{SC}}

\newcommand{\R}{\mathbb{R}}
\newcommand{\bound}{\mathcal{B}}
\newcommand{\betti}{B}
\newcommand{\fspace}{\mathcal{X}}
\newcommand{\feat}{\mathcal{F}}
\newcommand{\mI}{\mathcal{I}}

\newcommand{\TPCC}{\textsmaller{TPCC}}
\newcommand{\CW}{\textsmaller{CW}~}

\newcommand{\charEV}[1]{\widehat{e}_{#1}}

\DeclareMathOperator{\rk}{rk}

\newcommand\michael[1]{\noindent{\textcolor{magenta}{[MTS: #1]}}}
\newcommand\vincent[1]{\noindent{\textcolor{magenta}{[VPG: #1]}}}
\renewcommand\vincent[1]{}\renewcommand\michael[1]{}
\newtoggle{arxiv}
\toggletrue{arxiv}
\begin{document}
	\title{Topological Point Cloud Clustering}
	\author{Vincent P. Grande}
	\authornote{\textsmaller{VPG} acknowledges funding by the German Research Council (\textsmaller{DFG}) within Research Training Group 2236 (UnRAVeL)}
	\email{grande@cs.rwth-aachen.de}
	\affiliation{%
		\institution{\textsmaller{RWTH} Aachen University}
		\country{Germany}
	}
	
	\author{Michael T. Schaub}
	\authornote{\textsmaller{MTS} acknowledges partial funding from Ministry of Culture and Science (\textsmaller{MKW}) of the German State of North Rhine-Westphalia ("NRW R\"uckkehrprogramm") and the European Union (ERC, HIGH-HOPeS, 101039827). Views and opinions expressed are however those of the author(s) only and do not necessarily reflect those of the European Union or the European Research Council Executive Agency. Neither the European Union nor the granting authority can be held responsible for them.}
	\email{schaub@cs.rwth-aachen.de}
	\affiliation{%
		\institution{\textsmaller{RWTH} Aachen University}
		\country{Germany}
	}
% !TeX spellcheck = en_GB
\iftoggle{arxiv}{
\begin{teaserfigure}
		\includegraphics[width=\textwidth]{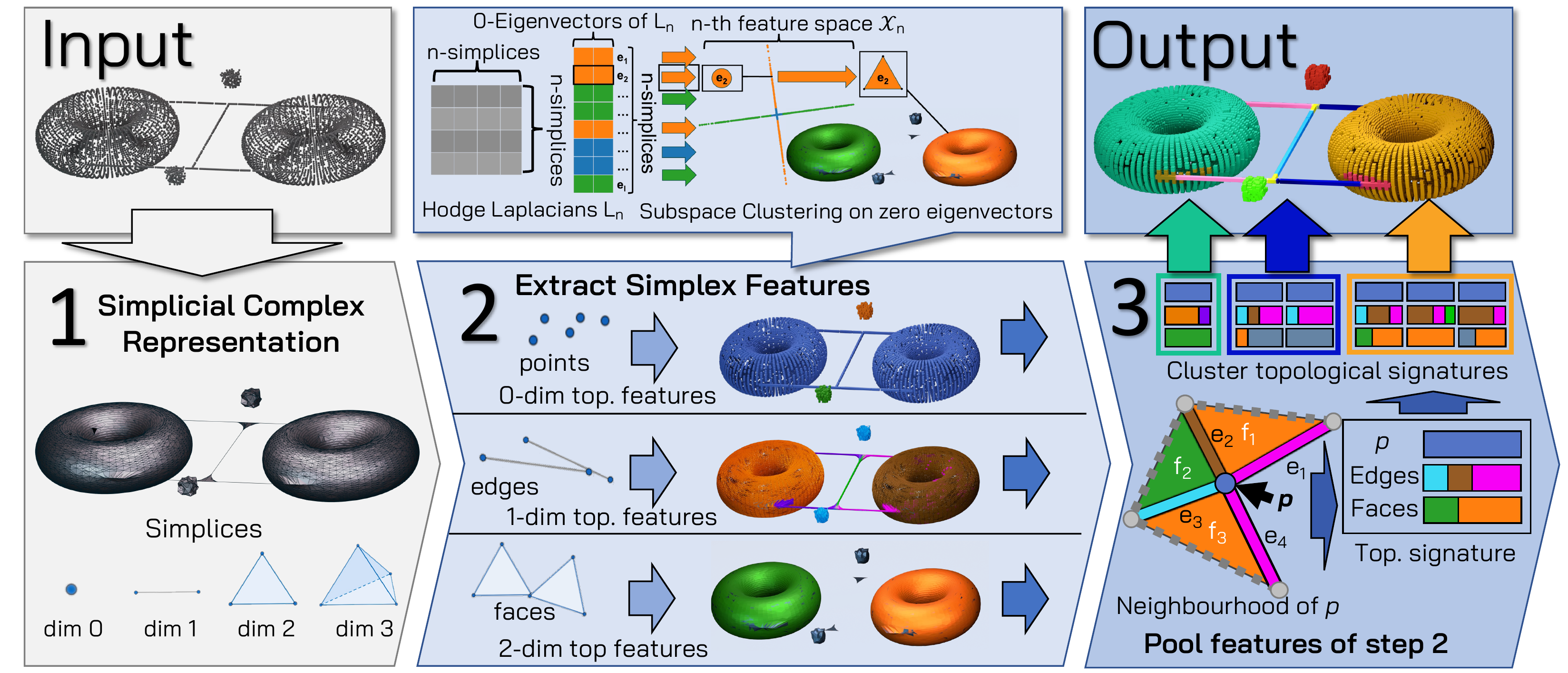}
	\caption{\textbf{Schematic of topological point cloud clustering (\TPCC{})}.
		\textbf{Step 1.} To capture the topological shape of the point cloud a \emph{simplicial complex} is constructed. 
		\textbf{Step 2.} Associated \emph{Hodge-Laplace} operators are constructed separately for each dimension. 
		The method extracts information from the sparse Hodge-Laplace operators by computing their $0$-eigenvectors. 
		The $0$-eigenvectors are indexed by the simplices in the respective dimensions. 
		We use these eigenvectors to embed the simplices into a single feature% vector
		space $\feat_n$ for each dimension of the simplices and perform subspace clustering on these feature spaces. 
		\textbf{Step 3.} For each simplex, we relay the clustering information back to its vertices. 
		Every point is thus equipped with a \emph{topological signature}, aggregating information on topological features over all dimensions. 
		Finally, the original points are clustered using a standard clustering approach on their topological signature.
		}
	\label{fig:fig1}\vspace{0.3cm}
\end{teaserfigure}
}{\begin{figure*}[ht!]
	\begin{center}
		\includegraphics[width=\textwidth]{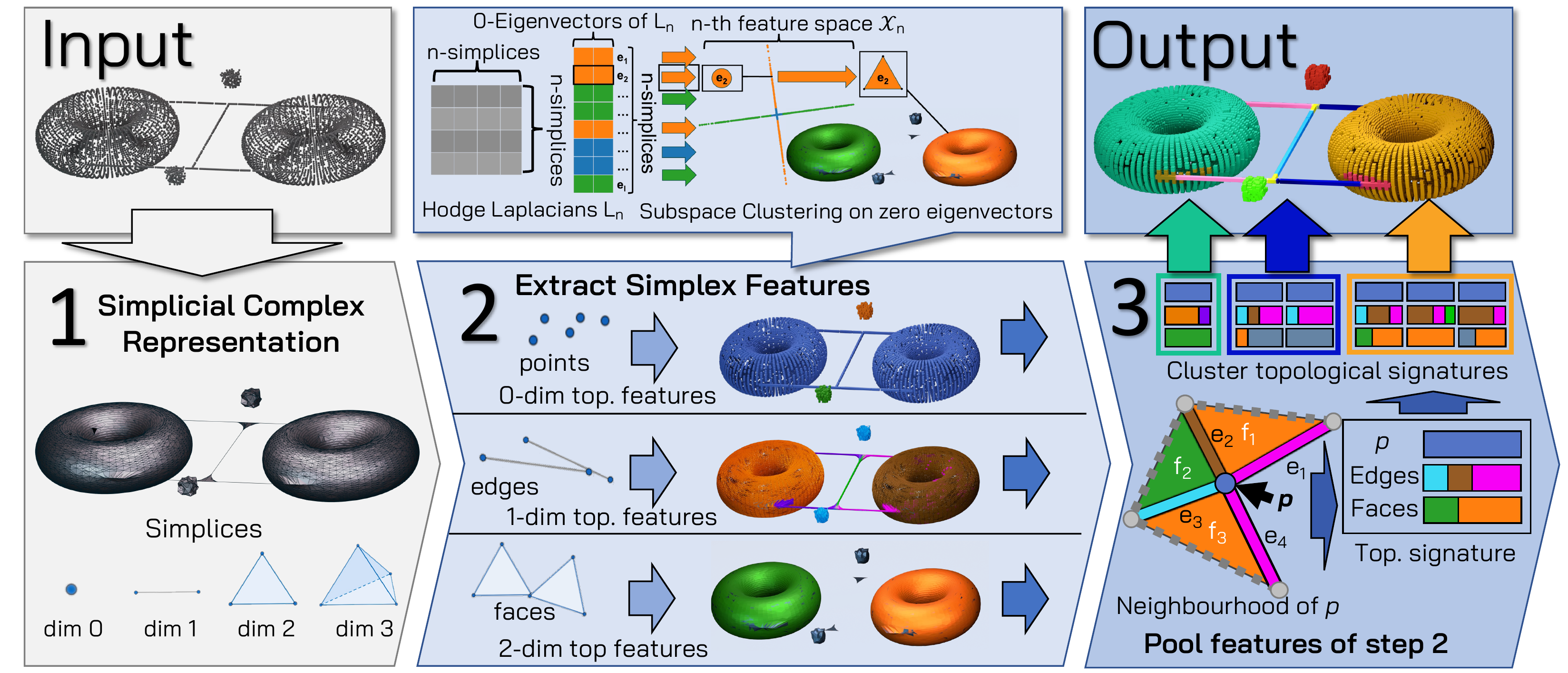}
		\caption{\textbf{Schematic of topological point cloud clustering (\TPCC{})}.
			\textbf{Step 1.} To capture the topological shape of the point cloud a \emph{simplicial complex} is constructed. 
			\textbf{Step 2.} Associated \emph{Hodge-Laplace} operators are constructed separately for each dimension. 
			The method extracts information from the sparse Hodge-Laplace operators by computing their $0$-eigenvectors. 
			The $0$-eigenvectors are indexed by the simplices in the respective dimensions. 
			We use these eigenvectors to embed the simplices into a single feature% vector
			space $\feat_n$ for each dimension of the simplices and perform subspace clustering on these feature spaces. 
			\textbf{Step 3.} For each simplex, we relay the clustering information back to its vertices. 
			Every point is thus equipped with a \emph{topological signature}, aggregating information on topological features over all dimensions. 
			Finally, the original points are clustered using a standard clustering approach on their topological signature.}
		\label{fig:fig1}
		\vspace{-0.5cm}
	\end{center}
\end{figure*}}
\begin{abstract}
	We present Topological Point Cloud Clustering (\TPCC{}), a new method to cluster points in an arbitrary point cloud based on their contribution to global topological features. 
	\TPCC{} synthesizes desirable features from spectral clustering and topological data analysis and is based on considering the spectral properties of a simplicial complex associated to the considered point cloud.
	As it is based on considering sparse eigenvector computations, \TPCC{} is similarly easy to interpret and implement as spectral clustering.
	However, by focusing not just on a single matrix associated to a graph created from the point cloud data, but on a whole set of Hodge~Laplacians associated to an appropriately constructed simplicial complex, we can leverage a far richer set of topological features to characterize the data points within the point cloud and benefit from the relative robustness of topological techniques against noise.
	We test the performance of \TPCC{} on both synthetic and real-world data and compare it with classical spectral clustering.
\end{abstract}
\iftoggle{arxiv}{
	\maketitle
	\section*{Code}
	The code to reproduce the experiments can be found \href{https://git.rwth-aachen.de/netsci/publication-2023-topological-point-cloud-clustering}{at gitlab}\footnote{\url{https://git.rwth-aachen.de/netsci/publication-2023-topological-point-cloud-clustering}}.
}

\section{Introduction}
A central quest of unsupervised machine learning and pattern recognition is to find meaningful (low-dimensional) structures within a dataset, where there was only apparent chaos before. 
In many cases, a data set consist of a point cloud in a high-dimensional space, in which each data point represents a real-world object or relation. 
Dimensionality reduction and clustering methods are thus often used as a first step towards extracting a more comprehensible description of the data at hand, and can yield meaningful insights into previously hidden connections between the objects.

The paradigm of most classical clustering algorithms assumes that there are only a few ``fundamental types'' within the observed data and every data point can be assigned to one of those types. 
How the notion of type is interpreted varies in different approaches, but in most cases, the data is assumed to be a disjoint union of these types plus noise, and the focus is on identifying an optimal \emph{local} assignment of the points to the respective types (clusters).
For instance, many prototypical clustering algorithms like $k$-means clustering~\cite{Steinhaus:1957} or mixture models like Gaussian mixtures~\cite{Day:1969} aim to group points together that are close according to some local distance measure in $\mathbb{R}^n$.
Other variants, like \textsmaller{DBSCAN}, aim to track dense subsets within the point cloud~\cite{Ester:1996}, and subspace clustering aims to find a collection of low-dimensional linear subspaces according to which the points can be grouped~\cite{Chen:2009}.
On the other hand, quantifying and utilizing the overall shape of the point cloud, i.e., how it is \emph{globally} assembled from the different clusters %(apart from that the data is assumed to correspond to a collection of clusters) 
or how to find the best possible cluster types to build up the data is typically not a concern.

In comparison, topological data analysis (\TDA{}), which has gained significant interest in the last decades \cite{Carlsson:2021}
emphasises an opposite perspective.
Here the dataset is effectively interpreted as one complex object, a topological space, whose ``shape'' we try to determine by measuring certain topological features (typically homology) to understand the \emph{global} make-up of the entire point cloud.
Such topological features are virtually omnipresent and are very flexible to describe highly complex shapes.
For instance, in medicine, they can measure the topology of vascular networks and can distinguish between tumorous and healthy cells~\cite{Stolz:2022}.
In public health studies, they have been used to analyse health care delivery network efficiency~\cite{Gebhart:2021},
and in Biochemistry, persistent homology has been used to analyse protein binding behaviour~\cite{Kovacev-Nikolic:2016}.
In Data Science, the Mapper algorithm uses topological features of data sets to produce a low dimensional representation of high dimensional data sets~\cite{Singh:2007}.

One key insight that has driven the success of the ideas of \TDA{} is that insightful higher-order information is often encoded in the topological features of (some amenable representation of) the data.
However, in contrast to classical clustering, the question of how the individual data points contribute to the make-up of the overall topological object is typically not a result of these types of analysis. 
This can render the interpretation of the results difficult, as often the individual data points have a concrete and meaningful (often physical) interpretation and we would thus like to know how these points relate to the overall measured topological object.

%Structures and relationships previously hidden to the eyes of $k$-means and such can only be made visible by considering the information that emerges when considering the point cloud as a whole.
%[This came after the realization that much important information on the objects represented by the data points is encoded in more complex global topological features.] 
%However, the go-to tools of \TDA{} provide the curious data scientist only with general information about the entire topological space, which proves to be difficult to interpret in applications. (Betti numbers, Persistence landscapes \cite{Bubenik2015}, need references)

The aim of this paper is to combine the advantages of these two perspectives and to establish a synthesis of traditional clustering algorithms with their easily interpretable output and the powerful notion of topological features of \TDA{}. 
Topological Point Cloud Clustering (\TPCC{}) bridges this gap between the local nature of classical clustering and the global features of \TDA{}, by aggregating information gained from multiple levels of a form of generalized spectral clustering on the $k$-simplices. 

\paragraph{Contributions}
We develop a novel topological point cloud clustering method that clusters the points according to what topological features of the point cloud they contribute to.
We prove that the clustering algorithm works on a class of synthetic point clouds with an arbitrary number of topological features across arbitrary dimensions. 
Finally, we verify the accuracy of topological point cloud clustering on a number of synthetic and real-world data and compare it with other approaches on data sets from the literature.

\paragraph{Organisation of the paper}
\label{sec:Organisation}
We introduce necessary topological notions in~\cref{sec:topology}. 
In \cref{sec:MainIdea}, we describe the main idea of topological point cloud clustering. 
A theoretical result on the accuracy of the algorithm on a class of synthetic point clouds is then presented in \ref{sec:theory}.
Finally, we show the distinguishing power of topological point cloud clustering on synthetic data, protein data and physically inspired real-world data in \cref{sec:experiments}. 
In particular, we compare the results of our algorithms with other clustering methods and study the robustness of \TPCC{} against noise on synthetic data.
Certain technical aspects of our algorithm and our running example are explained in more detail in \Cref{sec:Algorithm} and \Cref{sec:runningexamplelong}.

\paragraph{Related Work}
Our work builds upon several ideas that have been promoted in the literature.
In particular, \TPCC{} may be seen as a generalization of spectral clustering~\cite{vonLuxburg:2007}.
Spectral clustering starts with the construction of a graph from an observed point cloud, by identifying each data point with a vertex and connecting close-by points with an edge.
Vertices are then clustered according to their spectral embedding, i.e., the dominant eigenvectors of the graph representation considered (typically in terms of an associated Laplacian matrix).
However, these eigenvectors used by spectral clustering are merely related to connectivity properties (0-homology), and the produced clustering is thus restricted in terms of the topological features it considered.
Topological Mode Analysis \cite{Chazal2013} clusters point clouds using persistent homology. However, because only $0$-dimensional homology is considered the approach cannot cluster according to higher-order topological features like holes, loops and voids.

Our work does not just build a graph from the point cloud data but employs a simplicial complex to describe the observed point cloud (similar to how it is done in persistent homology) and embeds and clusters all $k$-simplices into the $0$-eigenvector space of the $k$-th Hodge Laplacian.
Related ideas of using embeddings based on the Hodge~Laplacian can be found in~\cite{schaub2020random,Chen2021,Ebli2019}:
The idea of defining a harmonic embedding to extract meaningful information about a simplicial complex has been discussed in the context of trajectory classification~\cite{schaub2020random, frantzen2021outlier}.
In \cite{Chen2021}, the authors study how this embedding is affected by constructing more complex manifolds from simpler building blocks.
However, they do not study how to decompose the underlying points based on this embedding.
In \cite{Ebli2019}, the authors develop a notion of harmonic clustering on the simplices of a simplicial complex. 
We use an extended version of this clustering as one step in \TPCC{}.
\cite{Krishnagopal2021} have as well considered harmonic clustering of simplices but only use it to detect large numbers of communities in small simplicial complexes.
In \cite{Perea2020}, the author uses a smoothed version of cohomology generators to quantify homology flows and build circular coordinates.
From a certain point of view, this is surprisingly similar to considering zero eigenvectors of Hodge Laplace operators.
Some related ideas to our work are also investigated in~\cite{Stolz2020}, where the authors provide a tool for detecting anomalous points of intersecting manifolds. 
As we will see, our algorithm is able to detect not only these points but can provide additional information about all remaining points as well.
There has been some work on surface and manifold detection in point clouds \cite{Martin2011, Hoppe:1992}.
In contrast to \TPCC{}, these algorithms don't provide any clustering or additional information on the points and are confined to manifold-like data, which is usually assumed to be a $2$-dimensional surface in $3$-dimensional space.
Approaches utilising tangent bundle constructions assume that the data corresponds to intersecting manifolds and that the desired clusters are represented by individual manifolds \cite{Wang2011, Gong2012, Tinarrage2023}. However, this may not be the case in real-world applications. \TPCC{} does not make such a restrictive assumption and is thus more widely applicable

The Hodge~Laplacian has also featured in a number of works from graph signal processing and geometric deep learning.
A homology-aware simplicial neural network is constructed in \cite{Keros2022}, extending previous models \cite{Roddenberry:2021,Bunch2020} on simplices of dimension two~\cite{Ebli2020,Bodnar2021}).
However, these approaches focus on a scenario where the higher-order simplices have some real-world meaning, e.g., 1-simplices can be identified by streets, neural links, or pairs of co-authors.
In contrast here our primary focus is on a scenario in which we are only given a point cloud to start with and thus only the points have a real-world meaning, whereas the higher dimensional features are added via some artificial simplicial complex simply to extract further information about the shape of the data. This is the case in most standard application scenarios.

\section{A Topological Notion of Features}
\label{sec:topology}
A main goal of topology is to capture the essence of spaces.
Topological tools try to describe globally meaningful features of spaces that are indifferent to local perturbations and deformations. 
This indifference of topological features to local perturbations can be a crucial asset when analysing large-scale datasets, which are often high-dimensional and noisy.
To leverage these ideas, we need to explain what we mean by \emph{topological features} throughout the paper.
A key assumption in this context is that high dimensional data sets may be seen as samplings from topological spaces --- most of the time, even low-dimensional manifolds \cite{Fefferman:2016}.
Rather than providing a complete technical account, in the following, we try to emphasize the relevant conceptual ideas and refer the interested reader to \cite{tomDieck:2008,Bredon:1993,Hatcher:2002} for further details.

\paragraph{Simplicial Complexes}
The prototypical topological space is a subset of $\R^n$ and hence continuous. 
Storing the infinite number of points in such a space individually is impossible.
On the other hand, our observed point cloud will always be discrete and non-connected. 
\emph{Simplicial complexes} (\tSC{}) bridge this gap between the continuous spaces of topology, and the discrete nature of our point cloud. 
They offer a way to build topological spaces from easy-to-define building blocks. 
Indeed, a well-known theorem in topology \cite{Quillen:1967} asserts that any topological space with the homotopy type of a \CW complex can be approximated by a simplicial complex.

%Formally, we define an \emph{abstract simplicial complex} as follows:
\begin{definition}[Abstract simplicial complex]
	An abstract simplicial complex $\SC$ consists of a set of vertices $X$ and a set of finite non-empty subsets of $X$, called simplices $S$, such that \textbf{(i)} $S$ is closed under taking non-empty subsets and \textbf{(ii)} the union over all simplices $\smash{\bigcup_{\sigma\in S} \sigma}$ is $X$.
	%\begin{enumerate}
	%	\item[(i)] $S$ is closed under taking non-empty subsets and
	%	\item[(ii)] the union over all simplices $\smash{\bigcup_{\sigma\in S} \sigma}$ is $X$.
	%\end{enumerate}
	For simplicity, we often identify $\SC$ with its set of simplices and use $\SC_n$ to denote the subset of simplicies with $n+1$ elements.
\end{definition}

Intuitively, in order to build a simplicial complex $\SC$, we first start with a set of vertices $V$. 
These are called the $0$-simplices. 
We can then add building blocks of increasing dimension. 
The $1$-simplices represent edges between $2$ vertices, the $2$-simplices are triangles between $3$ vertices that are already connected by edges. 
%We call the set of $n$-simplices the $n$-dimensional simplices because they resemble $n$-dimensional polyhedra.
An $n$-simplex resembles an $n$-dimensional polyhedra.
An $n$-simplex $\sigma_n$ connects $(n+1)$ vertices, given that they are already connected by all possible $(n-1)$-simplices. 
These $(n-1)$-simplices are then called the faces of $\sigma_n$.
We call two $(n-1)$-simplices \emph{upper-adjacent} if they are faces of the same $n$-simplex.
Correspondingly, we call two $n$-simplices \emph{lower-adjacent} if they share a common $(n-1)$-simplex as a face.

\paragraph{Vietoris--Rips complex}
%The final conceptional question to answer is how to turn the point cloud into a simplicial complex. 
Building the Vietoris--Rips complex is a method of turning a point cloud into a simplicial complex, approximating the topological features of the space it was sampled from. 
The Vietoris--Rips complex takes $2$ arguments as input: The point cloud $X$ and a minimal distance $\varepsilon$. 
It then builds a simplicial complex $\SC$ by taking $X$ as the set of vertices (and thus of $0$-simplices) of $\SC$.
Between every two distinct vertices of distance $d<\epsilon$ it adds an edge, i.e.~an $1$-simplex. 
Inductively, it then adds an $n$-simplex for each set of $(n+1)$ vertices in $X$ with pair-wise distance smaller than $\varepsilon$. In practice, one often restricts this process to simplices of dimension $n\le N$ for some finite number $N$.

\paragraph{Boundary matrices and the Hodge~Laplacians}

All topological information of a simplicial complex $\SC$ can be encoded in its \emph{boundary matrices} $\bound_n$. 
The rows of $\bound_n$ are indexed by the $n$-simplices of $\SC$, the columns are indexed by the $(n+1)$-simplices. 
\begin{definition}
	Let $\SC=(S,X)$ be a simplicial complex and $\preceq$ a total order on its set of vertices $X$. For $n\ge i$, $n\ge 1$ we define the $i$-th face map $f^n_i\colon\SC_n\rightarrow \SC_{n-1}$ by 
	\begin{align*}
		f^n_i&\colon \{x_0,x_1,\dots,x_n\}\mapsto \{x_0,x_1,\dots,\widehat{x}_i,\dots,x_n\}
	\end{align*}
	where we have that $x_0\preceq x_1\preceq\dots\preceq x_n$ and $\widehat{x}_i$ denotes the omission of $x_i$. Then we define the $n$-th boundary operator $\bound_n\colon \R[\SC_{n+1}]\rightarrow\R[\SC_{n}]$ by
	\[
		\bound_n\colon \sigma\mapsto \sum_{i=0}^{n+1}(-1)^if^{n+1}_i(\sigma).
	\]
	We identify $\bound_n$ with its matrix representation in lexicographic ordering of the simplex basis.
\end{definition}

Note that with this definition, $\bound_0$ is simply the familiar vertex-edge-incidence matrix of the associated graph built from the $0$- and $1$-simplices of~$\SC$.
\begin{definition}
	The $n$-th \emph{Hodge~Laplacian} $L_n$ of $\SC$ is a square matrix indexed by the $n$-simplices of $\SC$:
	\begin{equation}
		\smash{L_n\coloneqq\bound_{n-1}^\top \bound_{n-1}+\bound_n\bound_n^\top }
	\end{equation}
	where we take $\bound_{-1}$ to be the empty matrix. 
\end{definition}
The key insight about the $\bound_n$ is the following lemma:
\begin{lemma}%[\cite{Eckmann:1944,Friedman:1998}]
	For a simplicial complex $\SC$ with boundary matrices $\bound_i$ we have that
	%\begin{equation}
	$\smash{\bound_n\circ\bound_{n+1}=0}$
	%\end{equation}
	for $n\ge 0$.
\end{lemma}

\paragraph{Topological features: Homology and Betti numbers}
One of the main topological concepts is \emph{homology}. 
The $k$-th \emph{homology module} $H_k(X)$ of a space $X$ encodes the presence and behaviour of $k$-dimensional loops, enclosing generalised $(k+1)$-dimensional voids/cavities. 
The $k$-th \emph{Betti number} $\betti_k(X)$ of $X$ denotes the rank $\rk H_k(X)$ of the corresponding homology module. 
The $0$-th Betti number $\betti_0(X)$ is the number of connected components of $X$, $\betti_1(X)$ counts the number of loops and $\betti_2(X)$ counts how many $3$-dimensional cavities with $2$-dimensional borders are enclosed in $X$, and so on.

The following connection between the homology of an \tSC{} and its Hodge Laplacian will prove essential to us:
\begin{lemma}[\cite{Eckmann:1944,Friedman:1998}]
	For a simplicial complex $\SC$, let $L_n$ be the Hodge Laplacians and $\betti_n$ be the Betti numbers of $\SC$. Then we have that $
	\rk \ker L_n=\betti_n.$
\end{lemma}
The dimension of the kernel of the Hodge~Laplacian is equal to the number of orthogonal zero eigenvectors of $L_n$ over $\R$. 
Hence the Hodge~Laplacian provides a gateway for accessing topological features by computing eigenvectors.

\section{TPCC: Algorithm and Main Ideas}
\begin{algorithm}[tb]
	\caption{Topological Point Cloud Clustering (\TPCC{})}
	\label{alg:TPCC}
	\begin{algorithmic}
		\STATE {\bfseries Input:} Point cloud $X$, maximum dimension $d$
		\STATE Pick $\varepsilon$ and construct \textsmaller{VR}~complex $\SC$ of $X$
		\STATE Construct Hodge Laplacians $L_0,\dots,L_d$ of $\SC$
		\FOR{$i=0$ {\bfseries to} $d$}
		\STATE Compute basis $v_0^i,\dots, v_{\betti_i}^i$ of $0$-eigenvectors of $L_i$
		\STATE Subspace Clustering on rows of $\left[v_0^i,\dots, v_{\betti_i}^i\right]$
		\STATE Assign clusters to corresponding $i$-simplices of $\SC$
		\FOR{$x\in X$}
		\STATE (Top. signature of $x$:) Collect cluster information of $i$-simplices $\sigma_i$ with $x\in \sigma_i$
		\ENDFOR
		\ENDFOR
		\STATE Cluster $X$ according to topological signatures
		\STATE \textbf{Output:} Labels of $x\in X$
	\end{algorithmic}
\end{algorithm}
\label{sec:MainIdea}
In this section, we will describe Topological Point Cloud Clustering and its main ideas. A pseudocode version can be found in \Cref{alg:TPCC}.
\paragraph{Running example}
To illustrate our approach, we use the example displayed in \Cref{fig:fig1} consisting of two $4$-dimensional tori, depicted here in their projection to 3d space.
We connected the tori with two lines, which are again connected by a line.
Additionally, the point cloud includes two separate connected components without higher dimensional topological features.
Our point cloud has thus $11$ topological features across $3$ dimensions.
In terms of Betti numbers, we have $B_0=3$, $B_1=6$, and $B_2=2$.
For an in-depth discussion of the topology and construction of the running example, see \Cref{sec:runningexamplelong}.

\paragraph{Step 1: Approximating the space}
\iffalse
\begin{figure}[htb!]
	\vskip 0.2in
	\begin{center}
		\centerline{\includegraphics[width=\columnwidth]{figs/SCexample.png}}
		\caption{Simplicial Complex $\SC$ approximating the point cloud constructed from two tori. We can see the $1$-simplices marked in black, and the $2$-simplices marked in blue.}
		\label{fig:SCexample}
	\end{center}
	\vskip -0.2in
\end{figure}
\fi
To characterize our point cloud in terms of topological information, we suggest using the framework of simplicial complexes and the Vietoris--Rips Complex due to their straightforward definitions. 
The goal of this paper is to show that even with this naive approach of constructing a simplicial complex, a topologically meaningful clustering can be achieved. 
However, we note that \TPCC{} is agnostic towards the method the simplicial complex was constructed.
In low dimensions, the $\alpha$-complex provides a computationally efficient alternative with a lower number of simplices.
Complexes built using \textsmaller{DTM}-based filtrations are another alternative more robust to outliers \cite{Anai2020}.
%We say that a topological feature of the point cloud is of interest to us if it can be measured in real-valued homology. 
%This includes features like holes, loops, and multi-dimensional cavities, commonly found in real-world datasets and objects.

The general assumption is that the points of the point cloud are in some general sense sampled, potentially with some additional noise, from a geometrical space.
Now we would like to retrieve the topology of this original geometrical space from the information provided via the sampled points.
Hence, following common ideas within \TDA{}, we construct a computationally accessible topological space in terms of a simplicial complex on top of the point cloud approximating the ground truth space.
We denote the simplicial complex associated to our toy point cloud by~ $\SC$.
We note that the \TPCC{} framework works both with simplicial as well as with cellular complexes.
For simplicity however, we chose to stick with simplicial complexes throughout this paper.

\paragraph{Step 2A: Extracting topological features}
Having built the simplicial complex $\SC$, we need to extract its topological features. 
However, standard measures from topological data analysis only provide global topological features: 
For instance, Betti numbers are global features of a space, and persistence landscapes measure all features at once \cite{Bubenik2015}. 
In contrast, we are interested in how individual simplices and points are related to the topological features of the space. 
It is possible to extract a homology generator for a homology class in persistent homology \cite{Obayashi:2018}.
This approach is however not suitable for us, because the choice of a generator is arbitrary, and only the contribution of a small number of simplices can be considered. 

\TPCC{} utilises a connection between the simplicial Hodge-Laplace operators and the topology of the underlying \tSC{}. The dimension of the $0$-space of the $k$-th Hodge~Laplacian $L_k$ is equal to the $k$-th Betti number $\betti_k$  \cite{Eckmann:1944,Friedman:1998}.
%One advantage of using the Hodge Laplacian in homology computations to us is that 
Furthermore, the rows and columns of the Hodge~Laplacian $L_k$ are indexed by the $k$-simplices of $\SC$ and describe how simplices relate to each other, and in particular how they contribute to homology in terms of the null space of the $L_k$.

Let us now consider a concrete loop/boundary $\feat$ of an $(k+1)$-dimensional void.
We can then pick a collection $S$ of edges/$k$-simplices that represents this loop/boundary.
By assigning each simplex in $S$ the entry $\pm 1$ based on the orientation of the simplex, and every other simplex the entry $0$, we obtain a corresponding vector $e_S$.
The Hodge Laplace operator $L_k=\bound_{k-1}^\top \bound_{k-1}+\bound_k\bound_k^\top $ consists of two parts.
The kernel of the down-part, $\bound_{k-1}^\top \bound_{k-1}$, is spanned by representations of the boundaries of $(k+1)$-dimensional voids.
Hence, $e_S$ lies in this kernel: $\bound_{k-1}^\top \bound_{k-1}e_S=0$.
The kernel of the up-part of the Hodge Laplacian, $\bound_{k}\bound_{k}^\top $, is spanned by vectors that represent smooth flows along the $k$-simplices.
Thus by smoothing along the $k$-simplices one can turn $e_S$ into an eigenvector $\smash{\widehat{e}_S}$ of the entire Hodge Laplace operator $L_k$:
\begin{equation}
	L_k\widehat{e}_S=\bound_{k-1}^\top \bound_{k-1}\widehat{e}_S + \bound_{k}\bound_{k}^\top \widehat{e}_S=0.
\end{equation}
We call $\smash{\charEV{\feat}\coloneqq\widehat{e}_S}$ the \emph{characteristic eigenvector} associated to the loop/void $\feat$.

For simplicity, let us first consider the case where the $k$-th Betti number $B_k(\SC)$ is $1$. 
Then the zero-eigenvector $v_0$ of $L_k$ has one entry for every $k$-simplex and is the characteristic eigenvector $\charEV{\feat}$ for the single topological feature $\feat$ in dimension $k$. 
The entries of $v_0$ measure the contribution of the corresponding simplices to $\feat$. 
Intuitively, we can visualise the homology \enquote{flowing} through the simplices of the simplicial complex.
The entries of the eigenvector correspond to the intensity of the flow in the different $k$-simplices.
Because of the way we constructed $\charEV{\feat}$, the homology flow is then concentrated along the $k$-dimensional boundary of a hole/void in the space.
In the $1$-dimensional setting, this corresponds to harmonic flows along edges around the holes of an {\tSC{} \cite{Schaub:2021}.
	The case for the Betti number larger one $B_k>1$ will be discussed in more detail in the following paragraph.
		
	\paragraph{Step 2B: Clustering the $n$-simplices}
	Extending ideas from~\cite{Ebli2019,schaub2020random} we use the obtained coordinates for each simplex to cluster the simplices.
	In the case where $L_k$ has a single $0$-eigenvalue, we can easily cluster the simplices by simply looking at the entries of the $0$-eigenvector $e$:
	We can ignore the sign of the entry $e_\sigma$ of $e$ corresponding to a simplex $\sigma$ because this only reflects whether the arbitrarily chosen orientation of $\sigma$ aligns with the direction of the \enquote{homology flow}.
	Then, we assign all simplices $\sigma$ with absolute value of $e_\sigma$ above a certain threshold $|e_\sigma|>\varepsilon$ to the cluster of homologically significant simplices.
	The remaining simplices are assigned to a separate cluster.
	
	In the case of multiple boundaries of voids of the same dimension, i.e.~$B_k>1$, each boundary $\feat$ again corresponds to a \enquote{homology flow} with an associated characteristic eigenvector $\charEV{\feat_i}$ of $L_k$.
	The $\charEV{\feat_i}$ span the zero-eigenspace $E_k$ of $L_k$.
	However, an eigenvector solver will yield an arbitrary orthonormal basis $e_1,\dots,e_{B_k}$ of $E_k$ which is only unique up to unitary transformations.
	For a $k$-simplex $\sigma\in \SC_k$, let  $e_i(\sigma)$ denote the coordinate associated to $\sigma$ of the $i$-th basis vector $e_i$ of $E_k$ obtained by the eigenvector solver.  Now we denote by $\iota\colon \SC_k\rightarrow \R^{\betti_k}$,
	\[
	\iota\colon \sigma\mapsto \left(e_1(\sigma),e_2(\sigma),\dots,e_{B_k}(\sigma) \right)\in \R^{B_k}
	\]
	the embedding of the simplices into the $k$-th \emph{feature space} $\fspace_k\coloneqq\R^{\betti_k}$.
	Note that because we could have started with any orthonormal basis of $E_k$ the feature space is only defined up to arbitrary unitary transformations.
	The points of the feature space $\fspace_k$ represent different linear combinations of the basis vectors of the zero eigenspace of $L_k$. They also represent linear combinations of the $\charEV{\feat_i}$, and hence intuitively of the topological features.
	
	In the most simple case, the $\charEV{\feat_i}$ are orthogonal to each other and thus have disjoint support.
	Then they represent orthogonal linear combinations of the original basis of $E_k$ in the feature space $\fspace_k$.
	Hence the \enquote{natural} $\charEV{\feat_i}$-basis can be recovered by subspace clustering the $k$-simplices on the feature space $\fspace_k$ as depicted in the top of \Cref{fig:fig1}.
	For computational reasons, we subsample the simplices used for the subspace clustering.
	The remaining simplicies will then be classified using a $k$-nearest neighbour classifier on the feature space $\fspace_k$. 
	See \Cref{sec:linearcombinations} and \Cref{sec:multidimensional} for a discussion of more complicated special cases.
	\paragraph{Step 3A: Aggregating the information to the point level}
	
	\begin{figure}[tb!]
		\vskip 0.2in
		\begin{center}
			\centerline{\includegraphics[width=\columnwidth]{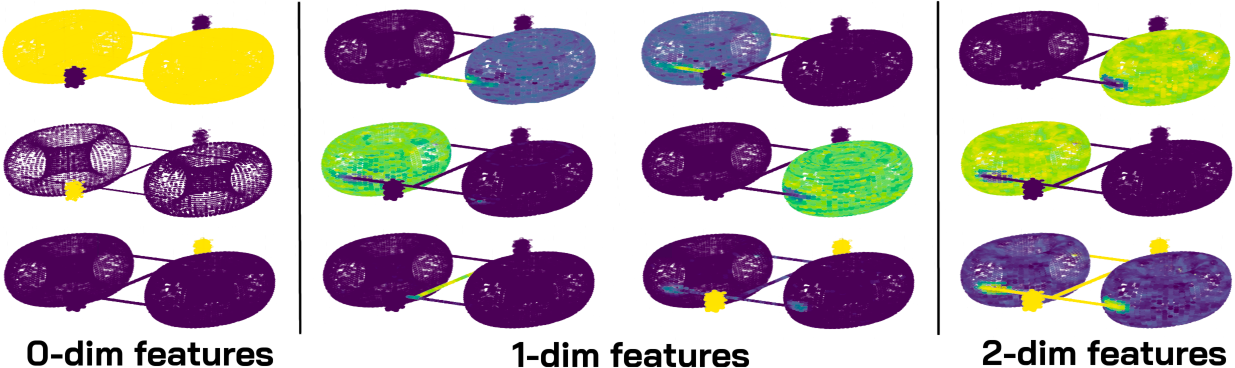}}
			\caption{
				Above we depict the heatmaps for all $16$ distinct combinations of topological features encoded in the topological signature across $3$ dimensions of our toy example. 
				Note that some of the features are redundant, as % for example 
				both edges and faces can measure membership of a torus.}
			\label{fig:topologicalFeatures}
		\end{center}
		\iftoggle{arxiv}{}{\vskip -0.2in}
	\end{figure}
	Finally, we can try to relate the information collected so far back to the points.
	For every point $x$ and every dimension $d$, we aggregate the cluster ids of the $d$-simplices which contain $x$.
	We call the collected information the \emph{topological signature} of $p$. 
	\begin{definition}[Topological Signature]
		\label{def:topsignature}
		Let $X$ be a point cloud with associated simplicial complex $\SC$. For a simplex $\sigma\in \SC$, we denote its cluster assignments from the previous step of \TPCC{} by $C(\sigma)$. Then, the \emph{topological signature} $\tau(x)$ of a point $x\in X$ is the multi-set
		\begin{equation*}%maybe inline?
			\tau(x)\coloneqq \{\!\{C(\sigma):\sigma \in \SC, x\in \sigma\}\!\}.
		\end{equation*}
		%\vincent{Alternatively, i could define it as a normalised vector. Which is better?}
	\end{definition}
	%\begin{definition}[Topological Signature II]
	%	\label{def:topsignature}
	%	Let $X$ be a point cloud with associated simplicial complex $\SC$. For an $i$-simplex $\sigma_i\in \SC_i$, we denote its cluster assignments from the previous step of TPCC by $C(\sigma_i)$. Then, the \emph{topological signature} $\tau(x)$ of a point $x\in X$ is a vector over 
	%	\begin{equation*}
		%		\tau(x)\coloneqq \{\!\{C(\sigma):\sigma \in \SC, x\in \sigma\}\!\}.%
		%	\end{equation*}
	%	\vincent{Alternatively, i could define it as a normalised vector. Which is better?}
	%\end{definition}
	After normalising for each $i$ by the number of $i$-simplices containing the point, topologically similar points will have a similar topological signature. 
	\Cref{fig:fig1}, Step 3 illustrates how the topological signature is calculated.
	In \Cref{fig:topologicalFeatures} we show how the different features of the topological signature highlight topologically different areas of the point cloud.
	Interestingly, we can even retrieve information on the gluing points between two topologically different parts.
	In \Cref{fig:finalclustering}, the \enquote{gluing points} between the tori and the lines receive their own cluster.
	This is because roughly half of the simplices adjacent to the gluing points receive their topological clustering information from the torus and the other half from the adjacent lines.
	Hence the gluing points are characterised by a mixture of different topological signatures.
	
	\paragraph{Step 3B: Computing the final clustering}
	\begin{figure}[tb!]
		\vskip 0.2in
		\begin{center}
			\centerline{\includegraphics[width=\columnwidth]{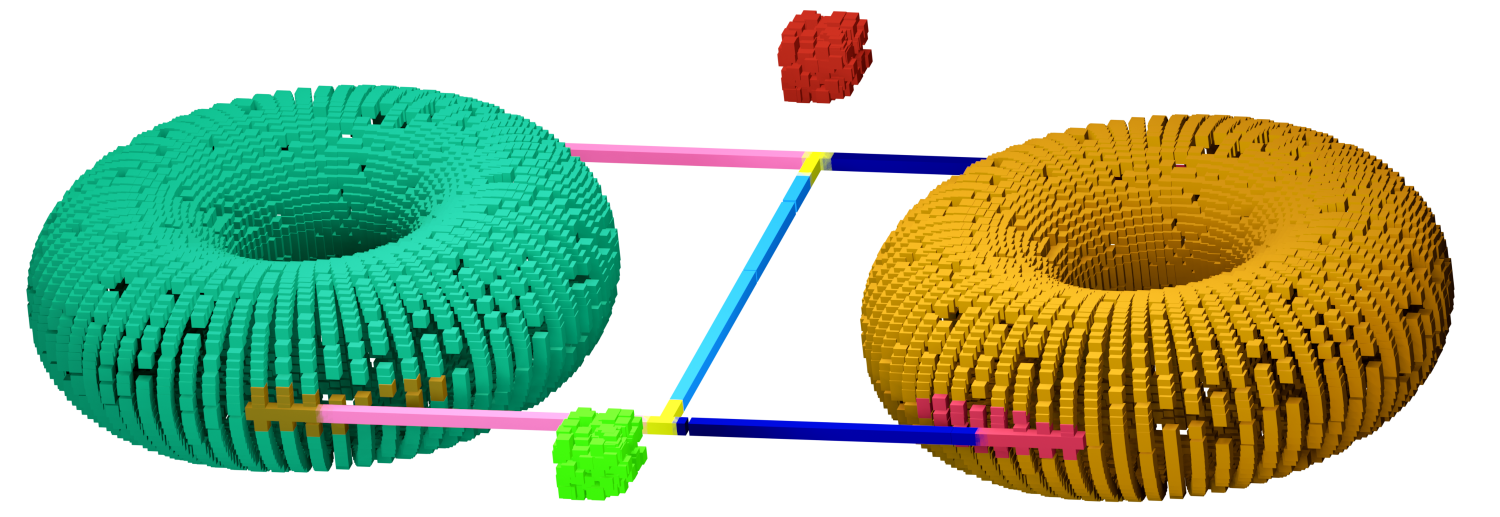}}
			\caption{The final clustering obtained with \TPCC{}. There are $10$ clusters in total. Two clusters identify the two tori (turquoise and ochre), two disconnected cubes (red and lime), dark blue and salmon for the connecting lines of the tori to the middle, azure for the middle line, yellow for the intersection of the lines, and fuchsia and brown for the gluing points of the points to the tori. 
				Note that there are virtually no outliers.
			}
			\label{fig:finalclustering}
		\end{center}
		\vskip -0.2in
	\end{figure}
	If we apply $k$-means or spectral clustering to a normalised form of the topological signatures of the points of our toy example, we arrive at the clustering of \Cref{fig:finalclustering}.
	
	In comparison to standard clustering methods, \TPCC{} can assign the same cluster to similar sets of points consisting of multiple connected components if they share the same topological features. 
	In \Cref{fig:finalclustering}, the two dark blue lines are assigned to the same cluster, because they both lie on the same loop and have no additional topological feature.
	This highlights the ability of \TPCC{} to take higher-dimensional information into consideration,exceeding the results obtainable by proximity-based information.
	
	\paragraph{Choice of parameters}
	\TPCC{} needs two main parameters, $\varepsilon$ and $d$.
	For the choice of the maximum homology degree $d$ to be considered there are three heuristics listed in decreasing importance: 
	
	\begin{enumerate}[I.]
	\item When working with real-world data, we usually know which kind of topological features we are interested in, which will then determine $d$.
	E.g., if we are interested in the loops of protein chains, we only need $1$-dimensional homology and thus choose $d=1$.
	When interested in voids and cavities in 3d tissue data, we need $2$-dimensional homology and thus choose $d=2$, and so on.
	\item There are no closed $n$-dimensional submanifolds of $\R^n$. 
	This means that if the point cloud lives in an ambient space of low dimension $n$, the maximum homological features of interest will live in dimension $n-1$ and hence we can choose $d=n-1$.
	\item In practice, data sets rarely have non-vanishing highly persistent homology in degree above $2$ and considering the dimensions $0$--$2$ usually suffices.
	Otherwise, one can calculate persistent homology up to the maximum computationally feasible degree to identify dimensions with sufficiently persistent homology classes, and then take $d$ as the maximum of these dimensions.
	\end{enumerate}

	Picking the correct value of $\varepsilon$ means choosing the correct scale.
	For the experiments in \Cref{fig:Robustness}, we have implemented a heuristic which computes the persistence diagram of the point cloud, and then picks the $\varepsilon$ maximizing the number of topological features with high persistence and minimizing the number of features with low persistence for this value.
	As can be seen, this method performs comparatively well for considerable noise.
	\paragraph{Technical considerations: Linear combinations of features}
	\label{sec:linearcombinations}
	\begin{figure}[tb!]
		\begin{center}
			\centerline{\includegraphics[width=\columnwidth]{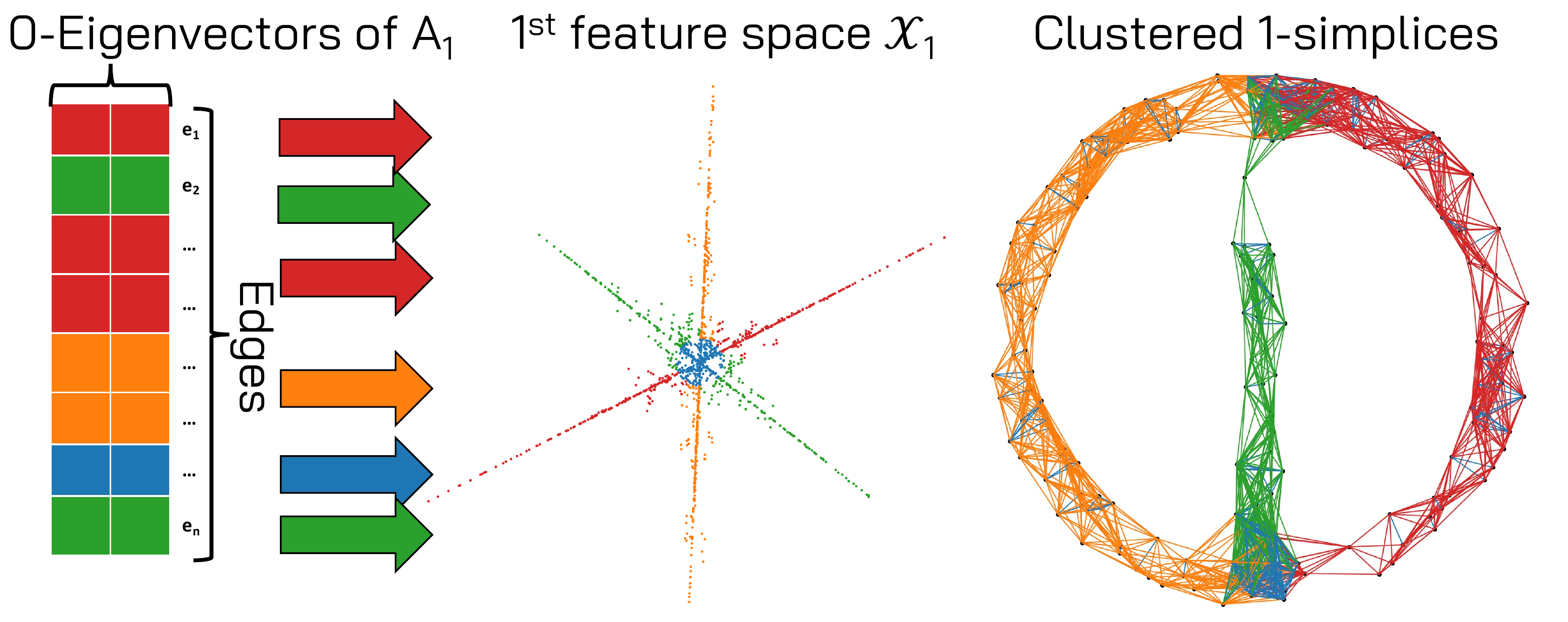}}
			\caption{The circle is divided into two parts by a vertical line. This gives the corresponding \tSC{} two generating loops in dimension $1$, corresponding to a $2$-dimensional $0$-eigenspace of the Hodge~Laplacian $L_1$ and a $2$-dimensional 1\textsuperscript{st} feature space $\fspace_1$. However, now there are three linear subspaces corresponding to linear combinations of the two generating loops. \TPCC{} is able to detect three different clusters of topologically significant edges.}
			\label{fig:SharedBoundary}
		\end{center}
		\iftoggle{arxiv}{}{\vskip -0.3in}
	\end{figure}
	In practice, topological features of the same dimension are not always separated in space.
	A bubble of soap may consist of two individual compartments divided by a thin layer of soap. 
	This middle layer then contributes to the boundaries of the two voids, i.e.~to two topological features of dimension $2$.
	How is this reflected in the $\charEV{\feat_i}$?
	
	This time, the characteristic eigenvectors $\charEV{\feat_i}$ corresponding to boundaries $\feat_i$ of voids of the same dimension are not orthogonal anymore.
	The supports of the $\charEV{\feat_i}$ overlap in the same simplices the corresponding boundaries $\feat_i$ overlap.
	In the feature space $\fspace_1$ of the example in \Cref{fig:SharedBoundary}, this is represented by the red, the green and the orange line having an approximate angle of $60^\circ$ to each other.
	The left loop is represented by an eigenvector $\charEV{\feat}$ with support on the green and orange edges, and vice-versa the right loop by $\charEV{\feat'}$ with support on the green and red edges.
	The homology flow on the middle line on the green edges is a linear combination of the homology flows of both generating loops.
	
	\section{Theoretical Guarantees for Synthetic Data}
	\label{sec:theory}
	In this section, we give a result showing that the algorithm works on a class of synthetic point clouds with an arbitrary number of topological features in arbitrary dimensions.
	The proof utilises the core ideas of the previous section.
	An easy way to realise a flexible class of topological space is to work with the wedge sum operator $\vee$ %. 
	%The wedge sum of two spaces is obtained by 
	gluing the two spaces together at a fixed base point. %, creating a single space.
	For $k>0$ and two topological spaces $X$ and $Y$ we have that $B_k(X\vee Y)=B_k(X)+B_k(Y)$.
	Hence the wedge sum combines topological features.
	%\vincent{Basically, we will just assume a perfect VR-complex of a bouquet of $k$-spheres. Then we can explicitly give eigenvectors corresponding to the features, giving us a very simple way to calculate the topological signatures, and then the perfect clustering.}
	\begin{theorem}
		\label{thm:guarantee}
		Let $\mathbb{P}\subset \R^n$ be a finite point cloud in $\R^n$ that is sampled from a space $X$. Furthermore, let $X=\smash{\bigvee_{i\in \mI}\mS_i^{d_i}}$
		%\[
		%X=\bigvee_{i\in \mI}\mS_i^{d_i}.
		%\]
		with finite indexing set $\mI$ with $\lvert\mI \rvert>1$ and $0<d\in \N$ be a bouquet of spheres . We assume that the geometric realisation of the simplicial approximation $\SC$ is homotopy-equivalent to $X$, and furthermore that the simplicial subcomplexes for the $\mS^{d_i}$ only overlap in the base-point, and divide $\mS^{d_i}$ into $d_i$-simplices.
		
		Then topological point cloud clustering recovers the different spheres and the base point accurately.
	\end{theorem}
	The full proof is given in \cref{sec:Proof}.
	\section{Numerical Experiments}
	\label{sec:experiments}
	\begin{table*}
		\centering
		\iftoggle{arxiv}{}{\scriptsize}
		\begin{tabular}{lrrrrrrrrr}
			\toprule
			&\TPCC& SpC&$k$-means&\textsmaller{OPTICS} & \textsmaller{DBSCAN}&\textsmaller{AC} &	Mean Shift&\textsmaller{AP}&	ToMATo\\
			\midrule
			$2$ spheres, $2$ circles (\Cref{fig:2spheres2circlesg})	&$\mathbf{0.97}$&$0.70	$&$0.48	$&$0.01	$&$0.00	$&$0.66	$&$0.84	$&$0.01	$&$0.90$\\
			Toy example (\Cref{fig:finalclustering})&$\mathbf{0.98}	$&$0.33	$&$0.28	$&$0.19	$&$0.11	$&$0.33	$&$0.81	$&$0.00	$&$0.91$\\
			Circle with line (\Cref{fig:SharedBoundary})&$\mathbf{0.85}	$&$0.23	$&$0.16	$&$0.11	$&$0.00	$&$0.25	$&$0.00	$&$0.23	$&$0.09$\\
			Sphere in circle, $\mathtt{noise}=0$ (\Cref{fig:Robustness} top)&$\mathbf{1.00}$	&$0.34	$&$0.02	$&$0.19	$&$0.00	$&$0.29	$&$0.00	$&$0.12	$&$0.06$\\
			Sphere in circle, $\mathtt{noise}=0.3$ (\Cref{fig:Robustness} bottom)& $\mathbf{0.53}$&$0.28$& $0.01$ & $0.22$ & $0.30$ & $0.27$ & $0.00$ & $0.13$ & $0.46$\\
			Energy landscape (\Cref{fig:ManifoldAnomalies} left)&$\mathbf{0.88}	$&$0.01	$&$0.01	$&$0.00	$&$0.00	$&$0.13	$&$0.00	$&$0.01$&$-0.02$\\
			\bottomrule
		\end{tabular}
		
		\caption{
			\textbf{Quantitative performance comparison of \TPCC{} with popular clustering algorithms.}
			We show the Adjusted Rand Index of \TPCC, Spectral Clustering (SpC), $k$-means, \textsmaller{OPTICS}, \textsmaller{DBSCAN}, Agglomerative Clustering (\textsmaller{AC}), Mean Shift Clustering, Affinity Propagation (\textsmaller{AP}), and Topological Mode Analysis Tool clustering (ToMATo) evaluated on six data sets. On every data set \TPCC{} performs best, indicating that the other algorithm are not designed for clustering points according to higher-order topological features.
		}
		
		\iftoggle{arxiv}{}{\vspace{-0.5cm}}
		\label{fig:quantitaiveComparison}
	\end{table*}
	\paragraph{Comparison with $k$-means and spectral clustering}
	We validated the effectiveness of \TPCC{} on a number of synthetic examples. In \Cref{fig:2spheres2circlesg}, we have clustered points sampled randomly from two spheres and two circles. 
	The algorithm recovers the spheres and circles.
	Normal (zero-dimensional) Spectral Clustering and $k$-means fail in choosing the right notion of feature, as the figure shows.
	For a visual comparison of \TPCC{} with other clustering algorithms on various datasets see \Cref{fig:comparisonScikitLearn} in the appendix.
	%In \Cref{fig:icosaeder}, we started with a graph. 
	%By constructing its clique complex, we could extract topological information useful for clustering it into five clusters: The two icosahedrons, the octahedron, and the points connected to both icosahedrons.
	
	\begin{figure}[htb!]
		\begin{center}
			\includegraphics[width=\columnwidth]{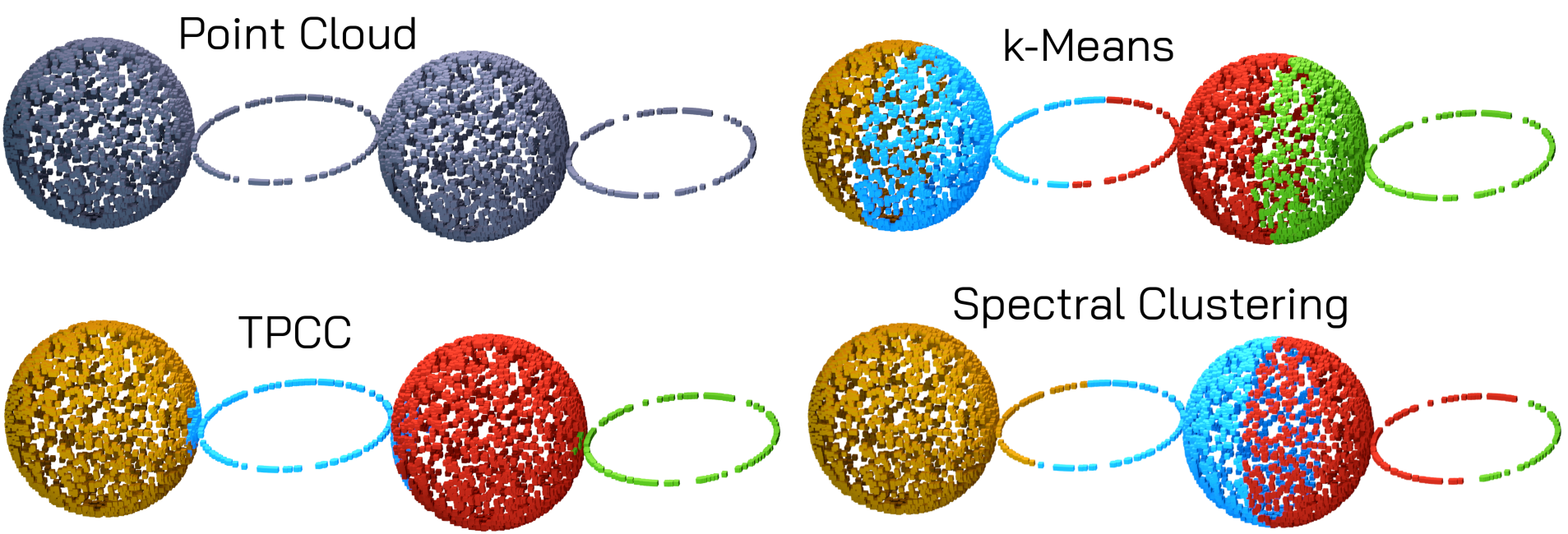}
			\caption{\TPCC{} is the only approach correctly distinguishing the spheres and circles.}
			\label{fig:2spheres2circlesg}
		\end{center}
	\end{figure}
	\label{sec:exp:synthetic}
	\paragraph{Comparison to Manifold Anomaly Detection}
	\begin{figure}[tb!]
		\begin{center}
			\begin{subfigure}{0.49\columnwidth}
				\includegraphics[width=\columnwidth]{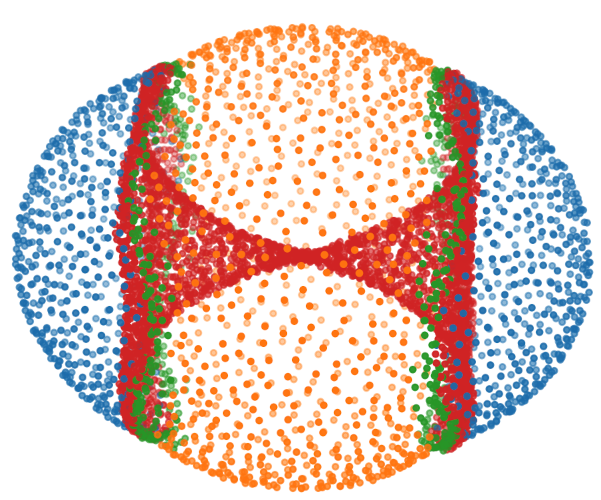}
			\end{subfigure}
			\begin{subfigure}{0.49\columnwidth}
				\includegraphics[width=\columnwidth]{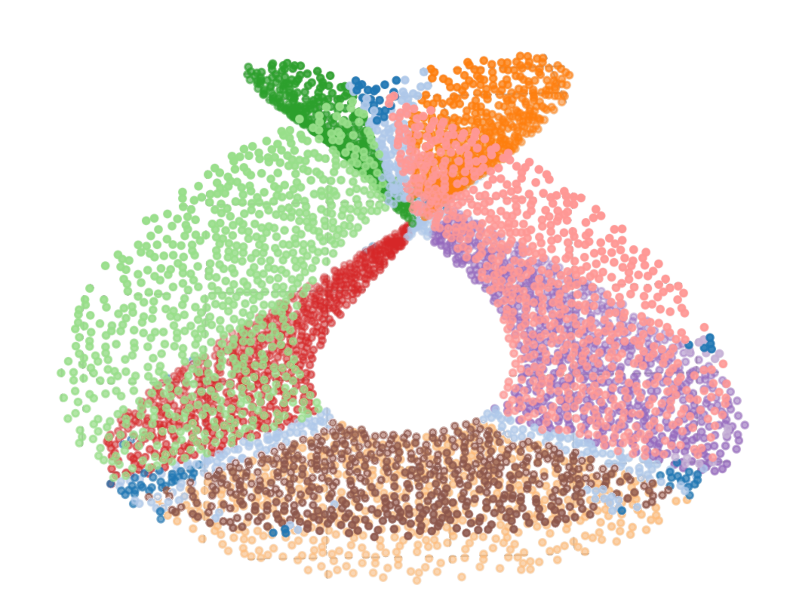}
			\end{subfigure}
			\caption{\emph{Left:} Energy landscape of cyclo-octane clustered by topological point cloud clustering. We have four different clusters, with the green one being the anomalous points.
				\emph{Right:} Clustering of the Henneberg surface.	
			}
			\label{fig:ManifoldAnomalies}
		\end{center}
		\iftoggle{arxiv}{}{\vskip -0.2in}
	\end{figure}
	
	\label{sec:exp:anomalies}
	In \cite{Stolz2020}, the authors propose a topological method for detecting anomalous points on manifolds. In \Cref{fig:ManifoldAnomalies} we use \TPCC{} on the same datasets \cite{Martin2011,Adams:2014} to show that our approach is also able to detect the anomalous points.
	Additionally, our method can classify the remaining points based on topological features.
	\paragraph{Experiments with Synthetic Data}
	\begin{figure}[htb!]
		\begin{center}
			\includegraphics[width=\columnwidth]{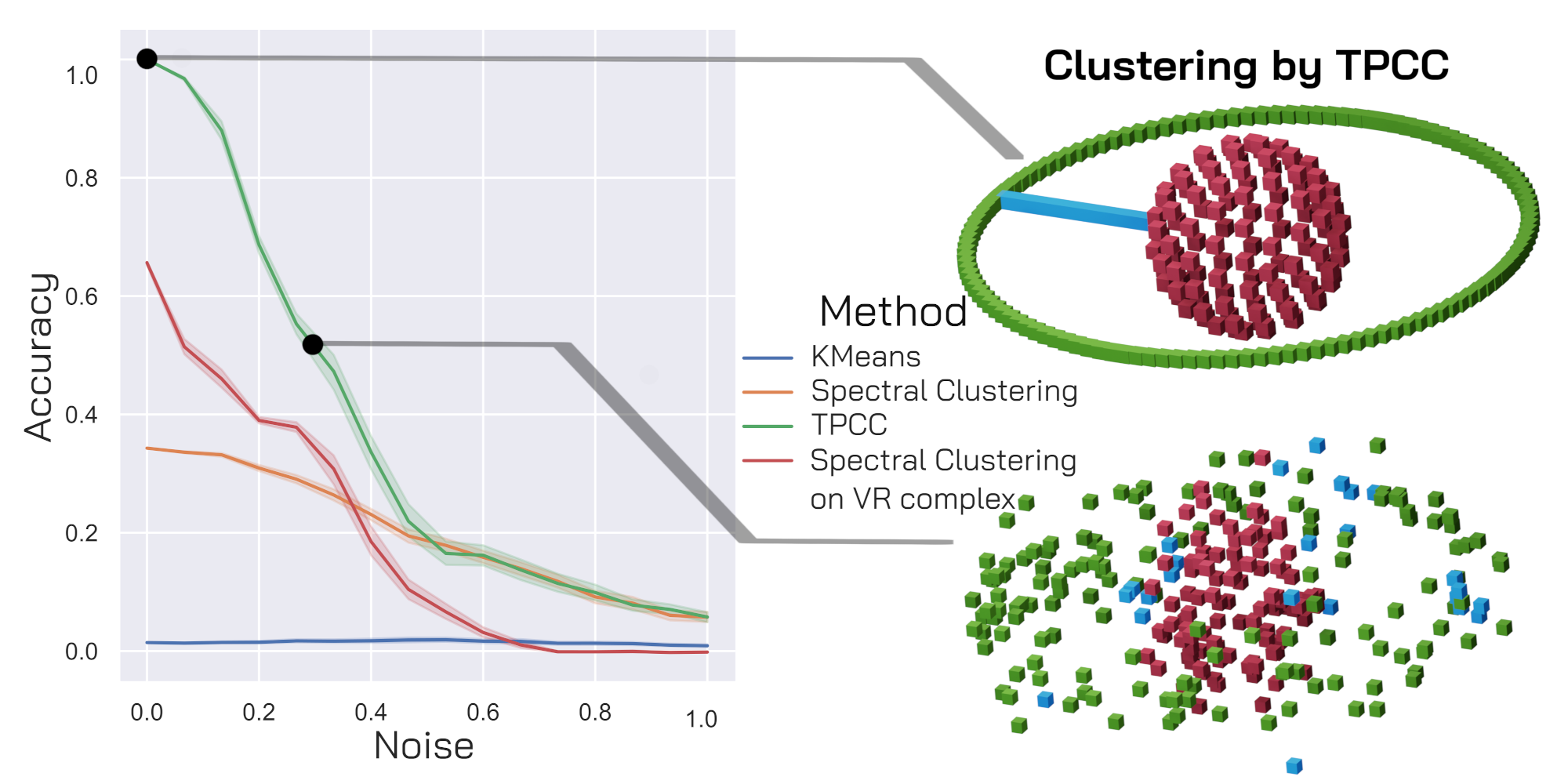}
			\caption{We have added i.i.d.~Gaussian noise with varying standard deviation specified by the parameter $\texttt{noise}$ on all three coordinates of every point. (For scale: The radius of the inner sphere is $1$.) \emph{Left:} Accuracy of \TPCC{}, $k$-Means and two versions of Spectral Clustering with increasing noise level. \emph{Spectral clustering} uses the radial basis affinity matrix, as implemented in scikit-learn. \emph{Spectral Clustering on \textsmaller{VR}~complex} uses the underlying graph of the simplicial complex used for \TPCC{}. Accuracy is measured by adjusted rand index and averaged over 100 samples. \emph{Right:} Example point clouds used for testing and clustering obtained by \TPCC{} for $\mathtt{noise}=0.0$ and $\mathtt{noise}=0.3$. 
			}%100 Tries each
			\label{fig:Robustness}
		\end{center}
		\iftoggle{arxiv}{}{\vskip -0.2in}
	\end{figure}
	As we make use of topological features, \TPCC{} is robust against noise by design. 
	We compare the accuracy of the clustering algorithm against $k$-means and spectral clustering on a point cloud consisting of a sphere, a circle, and a connecting line in \Cref{fig:Robustness}. 
	
	On low to medium noise levels, \TPCC{} significantly outperforms all other clustering methods.
	On higher noise levels, the topological features of the point cloud degenerate to features that can be measured by ordinary spectral clustering.
	Then, \TPCC{} and spectral clustering achieve similar accuracy scores.
	In \Cref{fig:Robustness} we see that already a noise setting of $\mathtt{noise}=0.3$ distorts the point cloud significantly, yet \TPCC{} still performs well.
	
	\paragraph{Proteins}
	Proteins are molecules that consist of long strings of amino acid residues.
	They play an integral role in almost every cellular process from metabolism, \textsmaller{DNA}~replication, to intra-cell logistics.
	Their diverse functions are hugely influenced by their complex 3d geometry, which arises by folding the chains of amino acid residues.
	The available data of protein sequences and 3d structure has increased dramatically over the last decades.
	However, functional annotations of the sequences, providing a gateway for understanding protein behaviour, are missing for most of the proteins.
	\cite{Smaili:2021} have shown that harnessing structural information on the atoms can significantly increase prediction accuracy of \textsmaller{ML}~pipelines for functional annotations.
	Thus being able to extract topological information on individual atoms of proteins is very desirable for applications in drug discovery, medicine, and biology.
	
	We tested \TPCC{} on \textsmaller{NALCN}~channelosome, a protein found in the membranes of human neurons \cite{Zhou:2022, Kschonsak:2022}. 
	The \textsmaller{NALCN}~channel regulates the membrane potential, enabling neurons to modulate respiration, circadian rhythm, locomotion and pain sensitivity. 
	It has a complex topological structure enclosing $3$ holes that are linked to its function as a membrane protein.
	
	The core idea is that when biological and topological roles correlate, \TPCC{} offers a way to better understand \emph{both}.
	\begin{figure}[tb!]
		\vskip 0.2in
		\begin{center}
			\includegraphics[width=\columnwidth]{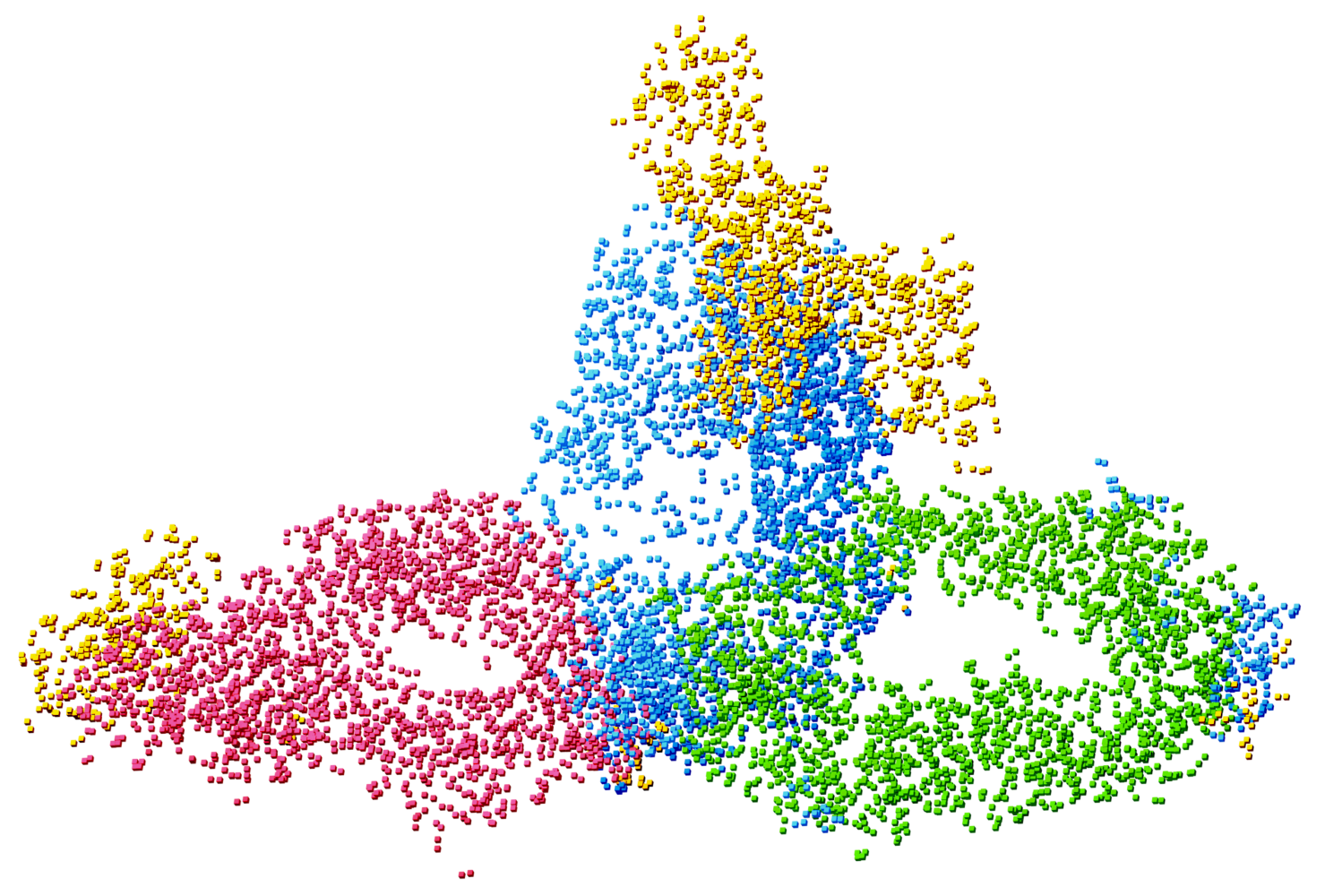}
			\caption{Clustered atoms of \textsmaller{NALCN}~channelosome. Points that border one of the holes are coloured red, blue, and green. The points without contribution to a loop are marked in yellow.}
			\label{fig:proteins}
		\end{center}
		\iftoggle{arxiv}{}{\vskip -0.2in}
	\end{figure}
	\label{sec:exp:proteins}
	\section{Discussion}
	\paragraph{Limitations}
	\TPCC{} can only cluster according to features that are visible to homology, e.g.~connected components, loops, holes, and cavities.
	For example, \TPCC{} cannot distinguish differently curved parts of lines or general manifolds.
	\TPCC{} constructs a simplicial complex (\textsmaller{SC}) to extract topological information
	Thus it needs to pick a single scale for every \textsmaller{SC}.
	If the topological information of the point cloud lie in different scales, \TPCC{} thus needs to do multiple feature aggregation steps for \textsmaller{SC}s of different scale.
	Finally, the points can be clustered according to the combined features.
	However, for each different scale the entire zero-eigenspace of the Hodge~Laplacian needs to be considered.
	Future work will focus on a method to cluster points based on the most persistent topological features across all scales.
	
	Persistent homology and the calculation of the zero eigenvectors of the Hodge Laplacian are computationally expensive and thus running \TPCC{} directly is not feasible on large data sets.
	However, usually the topological information can already be encoded in small subsets of the entire point cloud.
	In \Cref{table:Scalability} we show that \TPCC{} in combination with landmark sampling scales well for larger data sets while achieving high clustering performance.
	In addition, we believe that the main advantage of \TPCC{} is that it can do something no other existing point cloud clustering algorithm can do or was designed for, namely clustering points according to higher order topological features.
	Future work will focus on additionally improving efficiency by removing the need to compute the entire zero-eigenspace of the Hodge-Laplace operators.
	
	Because \TPCC{} uses persistent homology, it is robust against small perturbations by design.
	In \Cref{fig:Robustness} we analysed its clustering performance under varying levels of noise.
	However, with high noise levels, topological features vanish from persistent homology and thus \TPCC{} cannot detect them anymore.
	In future work, we try to take near-zero eigenvectors of the Hodge Laplacian into account, representing topological features contaminated by noise.
	This is similar to Spectral Clustering, where the near-zero eigenvectors represent almost-disconnected components of the graph.
	\paragraph{Conclusion}
	\TPCC{} is a novel clustering algorithm respecting topological features of the point cloud.
	We have shown that it performs well both on synthetic data and real-world data and provided certain theoretical guarantees for its accuracy.
	\TPCC{} produces meaningful clustering across various levels of noise, outperforming $k$-means and classical spectral clustering on several tasks and incorporating higher-order information.
	
	Due to its theoretical flexibility, \TPCC{} can be built on top of various simplicial or cellular representations of point clouds.
	Interesting future research might explore combinations with the mapper algorithms or cellular complexes.
	In particular, applications in large-scale analysis of protein data constitute a possible next step for \TPCC.
	\TPCC{} or one of its intermediate steps has potential as a pre-processing step for deep learning techniques, making topological information about points accessible for \textsmaller{ML} pipelines.
	\iftoggle{arxiv}{
	}{
\section*{Acknowledgments}
The authors thank the anonymous reviewers for helpful feedback that improved the paper.

\textsmaller{VPG} and \textsmaller{MTS} acknowledge funding by the \textsmaller{DFG} (German Research Foundation) – Research Training Group \textsmaller{UnRAVeL} 2236/2. \textsmaller{MTS} acknowledges funding by the Ministry of Culture and Science (\textsmaller{MKW}) of the German State of North
Rhine-Westphalia (\textsmaller{NRW} Rückkehrprogramm) and the European Union (\textsmaller{ERC}, \textsmaller{HIGH-HOPeS}, 101039827).
Views and opinions expressed are however those of \textsmaller{MTS} only and do not necessarily reflect those of the
European Union or the European Research Council Executive Agency; neither the European Union nor the
granting authority can be held responsible for them.
}
	\iftoggle{arxiv}{
	\bibliographystyle{acm}
}{
	\bibliographystyle{icml2023}	
}
\bibliography{refpc.bib}

\begin{thebibliography}{10}

\bibitem{Adams:2014}
{\scshape Adams, H., Tausz, A., and Vejdemo-Johansson, M.}
\newblock javaplex: A research software package for persistent (co)homology.
\newblock In {\em Mathematical Software -- ICMS 2014\/} (Berlin, Heidelberg,
  2014), H.~Hong and C.~Yap, Eds., Springer Berlin Heidelberg, pp.~129--136.

\bibitem{Anai2020}
{\scshape Anai, H., Chazal, F., Glisse, M., Ike, Y., Inakoshi, H., Tinarrage,
  R., and Umeda, Y.}
\newblock Dtm-based filtrations.
\newblock In {\em Topological Data Analysis: The Abel Symposium 2018\/} (2020),
  Springer, pp.~33--66.

\bibitem{Bodnar2021}
{\scshape Bodnar, C., Frasca, F., Wang, Y., Otter, N., Montufar, G.~F., Lio,
  P., and Bronstein, M.}
\newblock Weisfeiler and lehman go topological: Message passing simplicial
  networks.
\newblock In {\em International Conference on Machine Learning\/} (2021), PMLR,
  pp.~1026--1037.

\bibitem{Bredon:1993}
{\scshape Bredon, G., Ewing, J., Gehring, F., and Halmos, P.}
\newblock {\em Topology and Geometry}.
\newblock Graduate Texts in Mathematics. Springer, New York, 1993.

\bibitem{Bubenik2015}
{\scshape Bubenik, P., et~al.}
\newblock Statistical topological data analysis using persistence landscapes.
\newblock {\em J. Mach. Learn. Res. 16}, 1 (2015), 77--102.

\bibitem{Bunch2020}
{\scshape Bunch, E., You, Q., Fung, G., and Singh, V.}
\newblock Simplicial 2-complex convolutional neural networks.
\newblock In {\em NeurIPS Workshop: TDA {\&} Beyond\/} (2020).

\bibitem{Carlsson:2021}
{\scshape Carlsson, G., and Vejdemo-Johansson, M.}
\newblock {\em Topological Data Analysis with Applications}.
\newblock Cambridge University Press, 2021.

\bibitem{Chazal2013}
{\scshape Chazal, F., Guibas, L.~J., Oudot, S.~Y., and Skraba, P.}
\newblock Persistence-based clustering in riemannian manifolds.
\newblock {\em J. ACM 60}, 6 (nov 2013).

\bibitem{Chen:2009}
{\scshape Chen, G., and Lerman, G.}
\newblock Spectral curvature clustering (scc).
\newblock {\em International Journal of Computer Vision 81}, 3 (Mar 2009),
  317–330.

\bibitem{Chen2021}
{\scshape Chen, Y.-C., and Meila, M.}
\newblock The decomposition of the higher-order homology embedding constructed
  from the k-laplacian.
\newblock In {\em Advances in Neural Information Processing Systems\/} (2021),
  M.~Ranzato, A.~Beygelzimer, Y.~Dauphin, P.~Liang, and J.~W. Vaughan, Eds.,
  vol.~34, Curran Associates, Inc., pp.~15695--15709.

\bibitem{Day:1969}
{\scshape Day, N.~E.}
\newblock Estimating the components of a mixture of normal distributions.
\newblock {\em Biometrika 56}, 3 (1969), 463--474.

\bibitem{Ebli2020}
{\scshape Ebli, S., Defferrard, M., and Spreemann, G.}
\newblock Simplicial neural networks.
\newblock In {\em NeurIPS Workshop: TDA {\&} Beyond\/} (2020).

\bibitem{Ebli2019}
{\scshape Ebli, S., and Spreemann, G.}
\newblock A notion of harmonic clustering in simplicial complexes.
\newblock In {\em 2019 18th IEEE International Conference On Machine Learning
  And Applications (ICMLA)\/} (2019), pp.~1083--1090.

\bibitem{Eckmann:1944}
{\scshape Eckmann, B.}
\newblock {Harmonische Funktionen und Randwertaufgaben in einem Komplex.}
\newblock {\em Commentarii mathematici Helvetici 17\/} (1944/45), 240--255.

\bibitem{Ester:1996}
{\scshape Ester, M., Kriegel, H.-P., Sander, J., Xu, X., et~al.}
\newblock A density-based algorithm for discovering clusters in large spatial
  databases with noise.
\newblock In {\em KDD\/} (1996), pp.~226--231.

\bibitem{Fefferman:2016}
{\scshape Fefferman, C., Mitter, S., and Narayanan, H.}
\newblock Testing the manifold hypothesis.
\newblock {\em Journal of the American Mathematical Society 29}, 4 (Oct 2016),
  983–1049.

\bibitem{frantzen2021outlier}
{\scshape Frantzen, F., Seby, J.-B., and Schaub, M.~T.}
\newblock Outlier detection for trajectories via flow-embeddings.
\newblock In {\em 2021 55th Asilomar Conference on Signals, Systems, and
  Computers\/} (2021), IEEE, pp.~1568--1572.

\bibitem{Friedman:1998}
{\scshape Friedman, J.}
\newblock Computing betti numbers via combinatorial laplacians.
\newblock {\em Algorithmica 21}, 4 (Aug 1998), 331–346.

\bibitem{Gebhart:2021}
{\scshape Gebhart, T., Fu, X., and Funk, R.~J.}
\newblock Go with the flow? a large-scale analysis of health care delivery
  networks in the united states using hodge theory.
\newblock In {\em 2021 IEEE International Conference on Big Data (Big Data)\/}
  (Dec 2021), p.~3812–3823.

\bibitem{Gong2012}
{\scshape Gong, D., Zhao, X., and Medioni, G.}
\newblock Robust multiple manifolds structure learning.
\newblock In {\em Proceedings of the 29th International Coference on
  International Conference on Machine Learning\/} (Madison, WI, USA, 2012),
  ICML'12, Omnipress, p.~25–32.

\bibitem{Halko2011}
{\scshape Halko, N., Martinsson, P.~G., and Tropp, J.~A.}
\newblock Finding structure with randomness: Probabilistic algorithms for
  constructing approximate matrix decompositions.
\newblock {\em SIAM Review 53}, 2 (2011), 217--288.

\bibitem{NumPy}
{\scshape Harris, C.~R., Millman, K.~J., van~der Walt, S.~J., Gommers, R.,
  Virtanen, P., Cournapeau, D., Wieser, E., Taylor, J., Berg, S., Smith, N.~J.,
  Kern, R., Picus, M., Hoyer, S., van Kerkwijk, M.~H., Brett, M., Haldane, A.,
  del R{\'{i}}o, J.~F., Wiebe, M., Peterson, P., G{\'{e}}rard-Marchant, P.,
  Sheppard, K., Reddy, T., Weckesser, W., Abbasi, H., Gohlke, C., and Oliphant,
  T.~E.}
\newblock Array programming with {NumPy}.
\newblock {\em Nature 585}, 7825 (Sept. 2020), 357--362.

\bibitem{Hatcher:2002}
{\scshape Hatcher, A.}
\newblock {\em Algebraic Topology}.
\newblock Cambridge University Press, Cambridge, 2002.

\bibitem{Hoppe:1992}
{\scshape Hoppe, H., DeRose, T., Duchamp, T., McDonald, J., and Stuetzle, W.}
\newblock Surface reconstruction from unorganized points.
\newblock In {\em Proceedings of the 19th Annual Conference on Computer
  Graphics and Interactive Techniques\/} (New York, NY, USA, 1992), SIGGRAPH
  '92, Association for Computing Machinery, p.~71–78.

\bibitem{Keros2022}
{\scshape Keros, A.~D., Nanda, V., and Subr, K.}
\newblock Dist2cycle: A simplicial neural network for homology localization.
\newblock In {\em Proceedings of the AAAI Conference on Artificial
  Intelligence\/} (2022), vol.~36, pp.~7133--7142.

\bibitem{Kovacev-Nikolic:2016}
{\scshape Kovacev-Nikolic, V., Bubenik, P., Nikolić, D., and Heo, G.}
\newblock Using persistent homology and dynamical distances to analyze protein
  binding.
\newblock {\em Statistical Applications in Genetics and Molecular Biology 15},
  1 (Jan 2016).

\bibitem{Krishnagopal2021}
{\scshape Krishnagopal, S., and Bianconi, G.}
\newblock Spectral detection of simplicial communities via hodge laplacians.
\newblock {\em Physical Review E 104}, 6 (Dec 2021), 064303.

\bibitem{Kschonsak:2022}
{\scshape Kschonsak, M., Chua, H.~C., Weidling, C., Chakouri, N., Noland,
  C.~L., Schott, K., Chang, T., Tam, C., Patel, N., Arthur, C.~P., Leitner, A.,
  Ben-Johny, M., Ciferri, C., Pless, S.~A., and Payandeh, J.}
\newblock Structural architecture of the human nalcn channelosome.
\newblock {\em Nature 603}, 7899 (Mar 2022), 180–186.

\bibitem{ARPACK}
{\scshape Lehoucq, R., Sorensen, D., and Yang, C.}
\newblock {\em ARPACK users’ guide: Solution of large-scale eigenvalue
  problems with implicitly restarted arnoldi methods}.
\newblock Software, environments, tools. Society for Industrial and Applied
  Mathematics, 1998.

\bibitem{Martin2011}
{\scshape Martin, S., and Watson, J.-P.}
\newblock Non-manifold surface reconstruction from high-dimensional point cloud
  data.
\newblock {\em Computational Geometry 44}, 8 (2011), 427--441.

\bibitem{Obayashi:2018}
{\scshape Obayashi, I.}
\newblock Volume-optimal cycle: Tightest representative cycle of a generator in
  persistent homology.
\newblock {\em SIAM Journal on Applied Algebra and Geometry 2}, 4 (2018),
  508--534.

\bibitem{scikit-learn}
{\scshape Pedregosa, F., Varoquaux, G., Gramfort, A., Michel, V., Thirion, B.,
  Grisel, O., Blondel, M., Prettenhofer, P., Weiss, R., Dubourg, V.,
  Vanderplas, J., Passos, A., Cournapeau, D., Brucher, M., Perrot, M., and
  Duchesnay, E.}
\newblock Scikit-learn: Machine learning in {P}ython.
\newblock {\em Journal of Machine Learning Research 12\/} (2011), 2825--2830.

\bibitem{Perea2020}
{\scshape Perea, J.~A.}
\newblock Sparse circular coordinates via principal {$\Z$}-bundles.
\newblock In {\em Topological Data Analysis\/} (Cham, 2020), N.~A. Baas, G.~E.
  Carlsson, G.~Quick, M.~Szymik, and M.~Thaule, Eds., Springer International
  Publishing, pp.~435--458.

\bibitem{Quillen:1967}
{\scshape Quillen, D.~G.}
\newblock {\em Homotopical Algebra}, vol.~43 of {\em Lecture Notes in
  Mathematics}.
\newblock Springer, Berlin, 1967.

\bibitem{Roddenberry:2021}
{\scshape Roddenberry, T.~M., Glaze, N., and Segarra, S.}
\newblock Principled simplicial neural networks for trajectory prediction.
\newblock In {\em International Conference on Machine Learning\/} (2021), PMLR,
  pp.~9020--9029.

\bibitem{schaub2020random}
{\scshape Schaub, M.~T., Benson, A.~R., Horn, P., Lippner, G., and Jadbabaie,
  A.}
\newblock Random walks on simplicial complexes and the normalized hodge
  1-laplacian.
\newblock {\em SIAM Review 62}, 2 (2020), 353--391.

\bibitem{Schaub:2021}
{\scshape Schaub, M.~T., Zhu, Y., Seby, J.-B., Roddenberry, T.~M., and Segarra,
  S.}
\newblock Signal processing on higher-order networks: Livin’ on the edge...
  and beyond.
\newblock {\em Signal Processing 187\/} (2021), 108149.

\bibitem{Silva2004}
{\scshape Silva, V.~d., and Carlsson, G.}
\newblock {Topological estimation using witness complexes}.
\newblock In {\em SPBG'04 Symposium on Point - Based Graphics 2004\/} (2004),
  M.~Gross, H.~Pfister, M.~Alexa, and S.~Rusinkiewicz, Eds., The Eurographics
  Association.

\bibitem{Singh:2007}
{\scshape Singh, G., Mémoli, F., and Carlsson, G.}
\newblock Topological methods for the analysis of high dimensional data sets
  and 3d object recognition.
\newblock {\em Eurographics Symposium on Point-Based Graphics\/} (2007), 10.

\bibitem{Smaili:2021}
{\scshape Smaili, F.~Z., Tian, S., Roy, A., Alazmi, M., Arold, S.~T.,
  Mukherjee, S., Hefty, P.~S., Chen, W., and Gao, X.}
\newblock Qaust: Protein function prediction using structure similarity,
  protein interaction, and functional motifs.
\newblock {\em Genomics, Proteomics and Bioinformatics 19}, 6 (2021),
  998--1011.

\bibitem{Steinhaus:1957}
{\scshape Steinhaus, H.}
\newblock Sur la division des corps mat{\'e}riels en parties.
\newblock {\em Bull. Acad. Pol. Sci., Cl. III 4\/} (1957), 801--804.

\bibitem{Stolz:2022}
{\scshape Stolz, B.~J., Kaeppler, J., Markelc, B., Braun, F., Lipsmeier, F.,
  Muschel, R.~J., Byrne, H.~M., and Harrington, H.~A.}
\newblock Multiscale topology characterizes dynamic tumor vascular networks.
\newblock {\em Science Advances 8}, 23 (2022).

\bibitem{Stolz2020}
{\scshape Stolz, B.~J., Tanner, J., Harrington, H.~A., and Nanda, V.}
\newblock Geometric anomaly detection in data.
\newblock {\em Proceedings of the National Academy of Sciences 117}, 33 (2020),
  19664--19669.

\bibitem{gudhi:urm}
{\scshape {The GUDHI Project}}.
\newblock {\em {GUDHI} User and Reference Manual}.
\newblock {GUDHI Editorial Board}, 2015.

\bibitem{Tinarrage2023}
{\scshape Tinarrage, R.}
\newblock Recovering the homology of immersed manifolds.
\newblock {\em Discrete \& Computational Geometry\/} (2023), 1--86.

\bibitem{tomDieck:2008}
{\scshape tom Dieck, T.}
\newblock {\em Algebraic topology}, vol.~8.
\newblock European Mathematical Society, Z\"urich, 2008.

\bibitem{vonLuxburg:2007}
{\scshape von Luxburg, U.}
\newblock A tutorial on spectral clustering.
\newblock {\em Statistics and Computing 17}, 4 (Dec 2007), 395–416.

\bibitem{Wang2011}
{\scshape Wang, Y., Jiang, Y., Wu, Y., and Zhou, Z.-H.}
\newblock Spectral clustering on multiple manifolds.
\newblock {\em IEEE Transactions on Neural Networks 22}, 7 (2011), 1149--1161.

\bibitem{blendplot}
{\scshape Wells, C.}
\newblock blendplot, 2017.

\bibitem{Zhou:2022}
{\scshape Zhou, L., Liu, H., Zhao, Q., Wu, J., and Yan, Z.}
\newblock Architecture of the human nalcn channelosome.
\newblock {\em Cell Discovery 8}, 1 (Apr 2022), 33.

\bibitem{Zografos:2013}
{\scshape Zografos, V., Ellis, L., and Mester, R.}
\newblock Discriminative subspace clustering.
\newblock In {\em Proceedings of the IEEE Conference on Computer Vision and
  Pattern Recognition\/} (2013), pp.~2107--2114.

\end{thebibliography}
	\appendix
	\section{Implementation}
	\label{sec:Algorithm}	
	\begin{figure*}[tb!]
		\begin{center}
			\includegraphics[width=\textwidth]{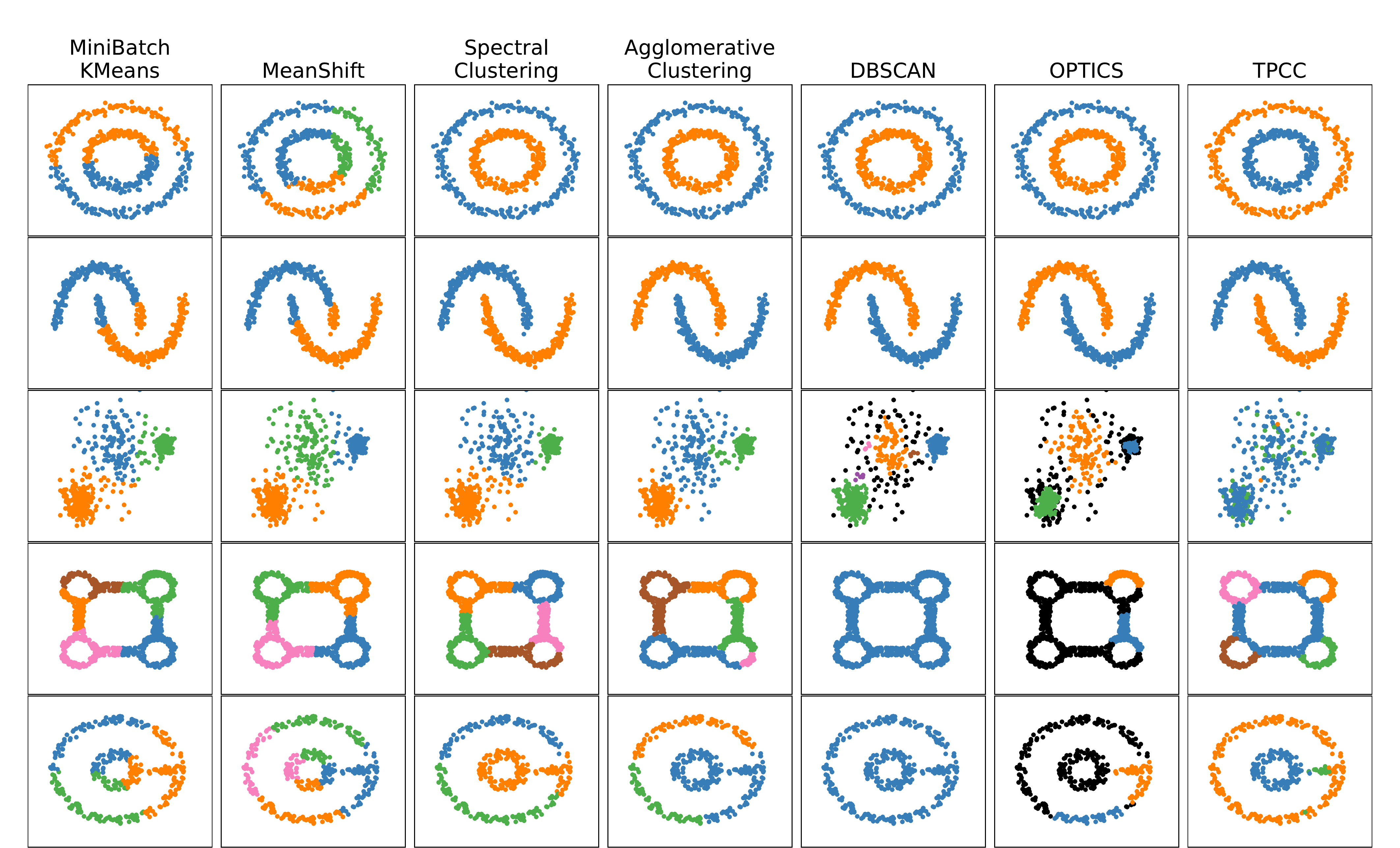}
			\caption{\textbf{Comparison of clusterings produced by \TPCC{} and other clustering algorithms on existing and new datasets.} While \TPCC{} is not able to find any structure in the third dataset, it identifies the topological substructures in the fourth and fifth dataset.
				(Cf. \texttt{scikit-learn}, \cite{scikit-learn}) While \TPCC{} is the only algorithm clustering the four connected circles respecting the topology, it still is inconsistent in whether to assign the shared parts of circles and rectangle to the cluster of the circle or the connecting lines. Theoretically, these shared parts would belong to 4 additional clusters. However, this would require the subspace clustering to identify $9$ different linear subspaces, which is not feasible with the implementation of subspace clustering we were using.}
			\label{fig:comparisonScikitLearn}
			\vspace{-0.5cm}
		\end{center}
	\end{figure*}
	To construct the simplicial complex used in \TPCC{} from a point cloud, we first computed a persistence diagram.
	Then we selected the parameter $\varepsilon$ in the range of the most persistent homology features. Hence, we connected all points $p_1$ and $p_2$ with $\|p_1-p_2\|_2<\varepsilon$.
	We also chose until which dimension we build the simplicial complex by looking at the topological features of the underlying point cloud.
	In practice, on all considered data sets the maximum dimension of topological features was $2$.
	Hence building the simplicial complex up to dimension $3$ suffices:
	We note that computing information on the $k$-th homology group and the $k$-th Betti number requires the simplices of dimension up to $k+1$.
	This is reflected in the shape of the $k$-th Hodge Laplacian $L_k\coloneqq\bound_{k-1}^\top \bound_{k-1}+\bound_k\bound_k^\top $ featuring $\bound_k$.
	The $k$-th boundary matrix $\bound_k$ maps $(k+1)$-simplices to $k$-simplices.
	\paragraph{Computational Complexity}
	
	\begin{table}
		\centering
		
		\scriptsize
		\begin{tabular}{ll|rr}
			\toprule
			& & \texttt{2spheres} & \texttt{6spheres} \\
			\midrule
			&Number of points&\num{4600}&\num{33600}\\
			\midrule
			\textbf{\TPCC} &Landmark sampling& \SI{0.7}{s} & \SI{48.7}{s} \\
			&Persistent homology& \SI{2.0}{s} &\SI{4.1}{s}\\
			&Eigenvector computation& \SI{3.6}{s} & \SI{31.4}{s}\\
			&Sum of times& \SI{6.3}{s}&$\mathbf{\SI{84.2}{s}}$\\
			&Adjusted Rand Index & \num{0.93} & \num{0.94} \\
			\midrule
			\textbf{\TPCC +witness} & Landmark sampling& \SI{0.7}{s} & \SI{48.7}{s} \\
			& Witness complex & \SI{0.2}{s}&\SI{615.5}{s}\\
			& Persistent homology & \SI{0.5}{s}& \SI{4.7}{s}\\
			& Eigenvector computation & \SI{5.4}{s}& \SI{19.7}{s}\\
			& Sum of times & \SI{6.8}{s}& \SI{688.6}{s}\\
			& Adjusted Rand Index & $\mathbf{\num{0.95}}$ &$\mathbf{\num{0.97}}$\\
			\midrule
			\textbf{SpC}& Time & $\mathbf{\SI{1.7}{s}}$ &\SI{346.3}{s} \\
			& Adjusted Rand Index & \num{0.71} & \num{0.47}\\
			\bottomrule
		\end{tabular}
		
		\caption{We test the scalability of clustering approaches using \TPCC{}. We compare the accuracy the running time and the Adjusted Rand Index (\textsmaller{ARI}) of \TPCC{} and Spectral Clustering on the data set of \cref{fig:2spheres2circlesg} and on a version with more points, spheres and circles. We also compare two different versions of constructing the \textsmaller{SC} in step 1 of \TPCC{}: We used a naive python implementation of min-max landmark sampling to select respectively \num{400} or \num{1200} landmarks. For the first version we constructed the Vietoris--Rips complex directly on this point cloud. For \TPCC +witness we constructed the witness complex based on the landmarks and the entire data set. After running \TPCC{}, we cluster the remaining points using a $1$-nearest neighbour approach. Both the \TPCC{} approaches achieved a significantly higher \textsmaller{ARI}{} than Spectral Clustering. On the large data set, \TPCC{} using landmark sampling and a \textsmaller{VR} construction had a significantly smaller running time then basic Spectral Clustering.
		}
		
		\iftoggle{arxiv}{}{\vspace{-0.5cm}}
		\label{table:Scalability}
	\end{table}
	
	Persistent homology and the calculation of the zero eigenvectors of the Hodge Laplacian are computationally expensive and thus \TPCC{} in its pure form does not scale well for large data sets.
	The complexity of random sparse eigensolvers is approximately $O(kT+k^2n)$ for $n\times n$ matrices where $k$ is the number of desired eigenvectors and $T$ is the number of flops required for one sparse matrix vector multiplication \cite{Halko2011}.
	The number of non-zero values in the $k$-th Hodge~Laplacian is bounded by the number of ordered pairs of upper-adjacent and of lower-adjacent $k$-simplices.
	For a fixed point density, fixed $\varepsilon$, and fixed $k$, the number of $k$-simplices $n$ is linear in the number of points.
	
	However, we believe that the main advantage of \TPCC{} is that it can do something no other existing point cloud clustering algorithm can do or was designed for, namely clustering points according to higher order topological features.
	
	Because \TPCC{} is agnostic to the type of simplicial complex constructed, its computational scalability can easily be improved by using a more efficient construction than Vietoris--Rips.
	Usually, the topological information of a data set is already contained in a small subset of the points.
	It is thus possible to use a witness complex construction \cite{Silva2004}, or to sample landmark points representing the topological structure, doing \TPCC{} on them, and then clustering the remaining points using k-nearest neighbours.
	
	In \Cref{table:Scalability}, we show that using \TPCC{} in conjunction with min-max landmark sampling and a $k$-nearest neighbour approach to classify the remaining points scales well to larger data sets while maintaining a high accuracy.
	
	\paragraph{Min-Max landmark sampling}
	Min-max landmark sampling provides a way to approximate the topology of a point cloud using a set of landmarks $L$.
	For a point cloud $X$, a distance function $d\colon X\times X\rightarrow \R_{\ge 0}$, and a desired number of landmarks $1\le k\le \lvert X\rvert$ we first sample a point $x\in X$ uniformly at random and add it to the set of landmarks $L$. Then we iteratively add the point $x\in X$ maximising the expression
	\[
	\min_{l\in L} d(l,x) 
	\]
	to $L$ until $\lvert L\rvert = k$. (For example, compare \cite{Perea2020}.)
	
	\paragraph{Witness complex}
	The (weak) witness complex, as introduced in \cite{Silva2004}, provides a way to approximate the topology of a large point cloud by a simplicial complex with significantly fewer vertices.
	It takes as input the original point cloud $X$, a set of landmarks $L$ in an ambient Euclidean space, and a parameter $R$ determining the length of edges.
	It then constructs a simplicial complex on $L$, where the simplices are added based on whether they are "witnessed" by points in $X$.
	While witness complexes are very robust topological approximators, their construction is computationally demanding for large point clouds $X$.
	
	\paragraph{Supplementary material}
	Code of our implementation to reproduce the experimental results will be made available in the supplementary material.%?
	\paragraph{Software used}
	We implemented the algorithm in python. We use the Gudhi library \cite{gudhi:urm} for all topology-related computations and operations. For general arithmetic and clustering purposes we use NumPy \cite{NumPy}, scikit-learn \cite{scikit-learn}, and ARPACK \cite{ARPACK}. For subspace clustering, we use DiSC \cite{Zografos:2013}. For 3d visualisation, we use blender with blendplot \cite{blendplot}.
	\section{Running Example}
	\label{sec:runningexamplelong}
	\begin{figure}[htb!]
		\begin{center}
			\centerline{\includegraphics[width=\columnwidth]{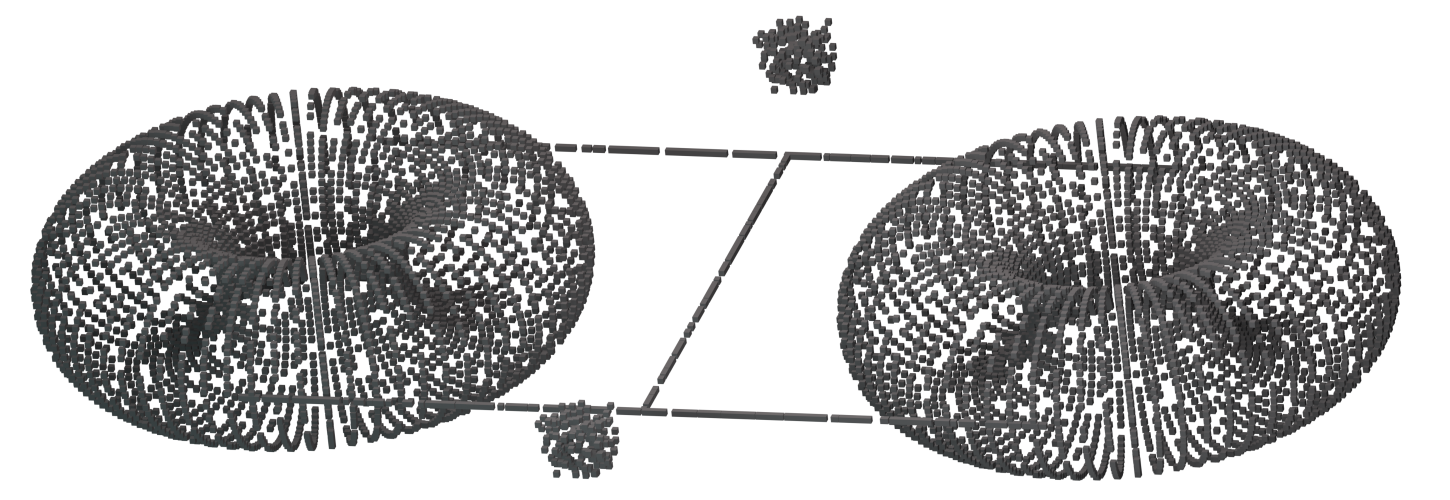}}
			\caption{The point cloud of our toy example projected to 3d space. Every point is represented by a small cube.}
			\label{fig:PointCloudExample}
		\end{center}
	\end{figure}
	Our toy example consists of two tori. 
	A torus can be seen as a doughnut where we removed the filling. 
	Topologically, it consists of a single connected component. 
	Hence its $0$-th Betti number $B_0$ is $1$, counting the number of connected components. 
	There are two main directions a $1$-dimensional loop can wrap around a torus. 
	First, there is the large loop going around the entire circle spanned by the torus. 
	Second, a loop can just wrap around a side of a torus.
	All other loops can be generated by concatenation of the previous two types of loops.
	Hence the $1$\textsuperscript{st} Betti number $B_1$ is $2$, counting the number of generating loops.
	Finally, there is a single $2$-dimensional cavity in a torus, representing the void left behind by removing the filling of the doughnut.
	Thus, the $2$\textsuperscript{nd} Betti number $B_2$ of a torus is $2$.
	We can embed a torus in $4$-dimensional space by taking it to be the product of two $1$-dimensional spheres.
	Note that we project the tori to $3$-dimensional space only for better readability in our plots.
	We sample the point cloud by first taking $5000$ points in a grid on each of the tori.
	We then randomly forget $20\%$ of the points in order to simulate noise.
	The tori are connected by two straight lines, from which we each sample $300$ points uniformly at random. 
	We connect the two lines by another straight line with $300$ randomly sampled points.
	The three lines add two more loops to the topological space
	Finally, we sample $200$ points uniformly at random from two cubes not connected with the rest of the topological space.
	Our point cloud has $11$ topological features across $3$ dimensions.
	In terms of Betti numbers, we have $B_0=3$, $B_1=6$, and $B_2=2$.
	
	\section{Technical considerations}
	\paragraph{Multi-dimensional subspaces of the feature spaces $\fspace_k$}
	\label{sec:multidimensional}
	\begin{figure}[htb!]
		\begin{center}
			\centerline{\includegraphics[width=\columnwidth]{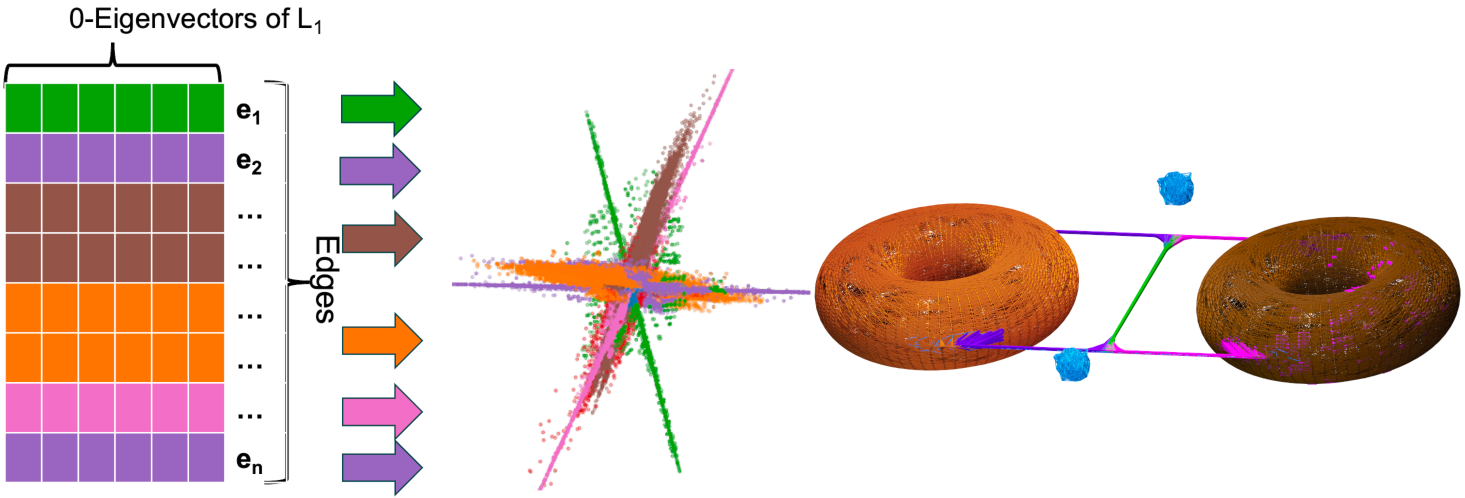}}
			\caption{We cluster the edges of the simplicial complex $\SC$ depicted in \Cref{fig:fig1}, our toy example. Its first Betti number $\betti_1(\SC)$ is $6$, corresponding to a $6$-dimensional zero-eigenspace of $L_1$. We show a projection of the $6$-dimensional feature space $\fspace_1$ to $3$-dimensional space. There are six different subspace clusters: three $1$-dimensional lines in purple, green, and pink corresponding to ordinary loops in the point cloud. Furthermore, there are two $2$-dimensional subspaces marked in brown and orange. They represent the edges in the two tori of our data set. Finally, there is one $0$-dimensional cluster corresponding to the edges without a contribution to homology in the two cubes marked in blue.}
			\label{fig:2dimsubspace}
		\end{center}
		\vskip -0.2in
	\end{figure}
	In practice, most of the subspaces of the feature space $\fspace_k$ used for subspace clustering are $1$-dimensional\footnote{Although we could regard the simplices without homological contribution as lying in a $0$-dimensional subspace.}.
	However, sometimes more complex substructures arise:
	Recall that our toy point cloud (\Cref{fig:fig1}, Input) consisted of two tori.
	The $1$-homology of each of the tori is generated by two loops, one of which follows the larger circle of the associated filled doughnut, the other one encircling a slice of the doughnut.
	Now imagine an edge $e$ starting in one of the outermost points of the torus.
	If the edge faces in a left or right direction, it will only contribute to the first loop. If it faces up or down, it only contributes o the second loop.
	However, an edge can point in an arbitrary superposition of the two directions.
	Thus also its homological contribution will be an arbitrary superposition of the two generating loops of the tori.
	In other words, the embeddings into the feature space $\iota(e)\in \fspace=\R^{B_k}$ of the edges $e$ running along the generating loops of the tori correspond to points on two orthogonal lines.
	The embeddings of all other edges on the surface of the torus lie on the $2$-dimensional subspace of $\fspace_1$ spanned by the two lines.
	Because the angles of the edges can vary continuously, the edges correspond to arbitrary points on the $2$-dimensional subspace.
	Thus, we propose clustering the edges based on membership in arbitrary-dimensional subspaces.
	In the toy point cloud example, we can hence measure to which torus an edge belongs by identifying the $2$-dimensional subspace its eigenvector coordinates lie on. 
	This is illustrated in \Cref{fig:2dimsubspace}.
	By allowing for detection of arbitrary dimensional subspaces of $\fspace_1$ our approach is able to detect significantly more topological features than the precursor approach in \cite{Ebli2019}.
	
	\paragraph{Extracting the dimensionality of topological features}
	\TPCC{} not only distinguishes different topological features, it is also capable of extracting additional information on the features. In particular, there are two ways the framework can distinguish between different dimensionalities of the features:
	
	\begin{enumerate}[I.]
		\item A $1$-dimensional loop will appear in the zero-eigenspace of the $1$-Hodge Laplacian, whereas a $2$-dimensional boundary of a void will appear in the $0$-eigenspace of the $2$-Hodge Laplacian.
		This information can easily be relayed back to the points.
		
		\item Topological features will manifest as linear subspaces of different dimensions of the zero-eigenspaces of the corresponding Hodge Laplace operators.
		Usually, these subspaces will be $1$-dimensional. 
		The subspace corresponding to the first homology group of the torus is however $2$-dimensional.
		(Cf. the previous paragraph.)
		This is because edges on the torus can point in arbitrary superpositions of the two "generating homological dimensions".
		(This, by Hurewicz's thm., corresponds to the respective generators of the fundamental group commuting with each other.)
		We can view this as the feature being another notion of $2$-dimensional and relay the information back to the points.
	\end{enumerate}
\section{Proof of Theorem \ref{thm:guarantee}}
\label{sec:Proof}
\begin{theorem*}
		Let $\mathbb{P}\subset \R^n$ be a finite point cloud in $\R^n$ that is sampled from a space $X$. Furthermore, let $X=\smash{\bigvee_{i\in \mI}\mS_i^{d_i}}$
		%\[
		%X=\bigvee_{i\in \mI}\mS_i^{d_i}.
		%\]
		with finite indexing set $\mI$ with $\lvert\mI \rvert>1$ and $0<d\in \N$ be a bouquet of spheres . We assume that the geometric realisation of the simplicial approximation $\SC$ is homotopy-equivalent to $X$, and furthermore that the simplicial subcomplexes for the $\mS^{d_i}$ only overlap in the base-point, and divide $\mS^{d_i}$ into $d_i$-simplices.
		
		Then topological point cloud clustering recovers the different spheres and the base point accurately.
\end{theorem*}
	\begin{proof}
	The $k$-th Betti number of $\SC$ is equal to the number of $i\in \mathcal{I}$ with $d_i=k$ (Cor. 2.25 \cite{Hatcher:2002}).
	Because spheres are orientable, we can simply assume that the $d_i$-simplices in $\mS_i^{d_i}$ are oriented such that each two adjacent $d_i$-simplices induce opposite orientations on the shared $(d_i)$-simplex.
	We now claim that for each $i\in \mI$ the indicator vector $e_i$ on the $d_i$-simplices in $\mS_i^{d_i}$ is an eigenvector of the $d_i$-th Hodge Laplacian $L_i$ of $\SC$.
	Because of our assumption on $\SC$, there are no $(d_i+1)$-simplices upper-adjacent to the $d_i$-simplices of $\mS_i^{d_i}$.
	Hence, we obtain the first half of our claim, namely that $\bound_{d_i}\bound_{d_i}^\top e_i=0$ holds.
	We have assumed that $\SC$ was constructed in such a way that each $(d_i-1)$-simplex $\sigma_{d_i-1}$ of $\smash{\mS_i^{d_i}}$ has exactly two upper-adjacent neighbours $\sigma^1_{d_i}$ and $\sigma^2_{d_i}$.
	Because $\smash{\sigma^1_{d_i}}$ and $\smash{\sigma^2_{d_i}}$ induce the opposite orientation on $\sigma_{d_i-1}$, the corresponding entries of the $(d_i-1)$-th boundary matrix $\bound_{d_i-1}$ of $\SC$ are $1$ and $-1$.
	Thus we also have $\bound_{d_i-1}e_i=0$ and finally $\smash{L_{d_i}e_i=\bound_{d_i}\bound_{d_i}^\top e_i+\bound_{d_i-1}^\top \bound_{d_i-1}e_i=0}$.
	This proves the claim.
	
	The eigenvectors $e_i$ of the same dimension are orthogonal and match in number with the corresponding Betti number of $\SC$. 
	Hence the $e_i$ span the eigenspaces of the Hodge Laplace operators of $\SC$.
	For all $i\in\mI$ the entries of the $d_i$-simplices in $\mS_i^{d_i}$ in the matching zero eigenvectors $e_j$ are $1$ for $j=i$, and $0$ else.
	All other $d$-simplices for $d>0$ have trivial eigenvector entries.
	Thus, subspace clustering recovers the top-level simplices in each of the spheres and assigns every other simplex to the trivial homology cluster.
	The topological signature of the points in the sphere $\mS^{d_i}_i$ in dimension $d_i$ will then feature a characteristic cluster of $(d_i)$-simplices and a trivial signature across the other dimensions.
	Finally, the topological signatures of the base point will feature all characteristic clusters.
	Hence $k$-means on the topological signatures can distinguish the points on the different spheres and the base point.
\end{proof}

\end{document}